%% file: noise_arxiv.tex
\documentclass{article}
\usepackage[utf8]{inputenc}
\usepackage[margin=1in]{geometry}
\usepackage{amsmath} % assumes amsmath package installed
\usepackage{todonotes}
\usepackage{amsthm}
\input{macro/jared_macro.tex}

\input{macro/noise_in_var_macro.tex}
\title{\LARGE \bf
Data-Driven Stabilizing and Robust Control \\ of Discrete-Time Linear Systems \\ with Error in Variables
}

\author{Jared Miller$^1$, Tianyu Dai $^1$, Mario Sznaier$^1$
\thanks{$^1$J. Miller, T. Dai, and M. Sznaier are with the Robust Systems Lab,  ECE Department, Northeastern University, Boston, MA 02115. (e-mails: \{miller.jare, dai.ti\}@northeastern.edu, msznaier@coe.neu.edu).}
\thanks{J. Miller, T. Dai, and M. Sznaier were partially supported by NSF grants  CNS--1646121, ECCS--1808381 and CNS--2038493, AFOSR grant FA9550-19-1-0005, and ONR grant N00014-21-1-2431.
J. Miller was in part supported by the Chateaubriand Fellowship of the Office for Science \& Technology of the Embassy of France in the United States, AFOSR grant FA9550-19-1-0005, and the AFOSR International Student Exchange Program.}}

\begin{document}

\maketitle
\thispagestyle{empty}
\pagestyle{empty}

%%%%%%%%%%%%%%%%%%%%%%%%%%%%%%%%%%%%%%%%%%%%%%%%%%%%%%%%%%%%%%

\input{sections/abstract}
\input{sections/introduction}
\input{sections/contributions_arxiv}
\input{sections/summary_arxiv}
\input{sections/preliminaries}

\input{sections/full_method}
\input{sections/alternatives_method}
\input{sections/optimal_control}

\input{sections/complexity}

\input{sections/all_noise}
\input{sections/examples}
\input{sections/extensions_arxiv}

\input{sections/Conclusion}
% \appendices
\input{sections/acknowledgements}

\bibliographystyle{IEEEtran}
\bibliography{noise_references.bib}
\appendix
\input{sections/app_sos}

\input{sections/app_continuity}
\input{sections/app_polynomial_sw}

\end{document}

%% file: macro/jared_macro.tex
\usepackage{xcolor}
\usepackage{amsmath}
\usepackage{amssymb}
\usepackage{mathtools}
\usepackage{subcaption}
\usepackage{hyperref}
\hypersetup{colorlinks=true, linkcolor=black}
\usepackage{algorithm}
\usepackage{algpseudocode} 
\usepackage{mathrsfs}
\usepackage{cite}
\usepackage{adjustbox}
\usepackage{float}

\DeclareMathOperator*{\find}{find}
\newcommand{\R}{\mathbb{R}} 
\newcommand{\K}{\mathbb{K}} 
 
\newcommand{\N}{\mathbb{N}}

\newcommand{\0}{\mathbf{0}}
\newcommand{\1}{\mathbf{1}}
\newcommand{\psd}{\mathbb{S}}

\DeclarePairedDelimiter{\abs}{\lvert}{\rvert}
\DeclarePairedDelimiter{\norm}{\lVert}{\rVert}
\DeclarePairedDelimiter{\ceil}{\lceil}{\rceil}
\DeclarePairedDelimiter{\Tr}{\textrm{Tr}(}{)}

\DeclarePairedDelimiter{\diag}{\textrm{diag}(}{)}

\DeclarePairedDelimiterX{\inp}[2]{\langle}{\rangle}{#1, #2}

\DeclareMathOperator*{\argmin}{arg\!\,min}

\newtheorem{thm}{Theorem}[section]
\newtheorem{lem}{Lemma}[section]

\newtheorem{prob}{Problem}
\newtheorem{assum}{Assumption}

\newtheorem{rem}{Remark}

%% file: macro/noise_in_var_macro.tex
\newcommand{\dx}{\Delta x} 
\newcommand{\du}{\Delta u} %corrected by MS
\newcommand{\dc}{\mathcal{D}}
\newcommand{\pc}{\mathcal{P}}
\newcommand{\bpc}{\bar{\mathcal{P}}}

\DeclarePairedDelimiter{\mabs}{\textbf{Mabs}(}{)}

\newcommand{\all}{\textrm{all}}
\newcommand{\alt}{\textrm{alt}}

\usepackage{acronym}
\usepackage{circuitikz}

\newcommand{\tz}{\tilde{\zeta}}
\newcommand{\tm}{\tilde{\mu}}
\newcommand{\tph}{\tilde{\phi}}

%% file: sections/abstract.tex
\begin{abstract}
\label{sec:abstract}
% Text in this \old{teal color} is copied from the conference paper (in the interests of formatting and layout), and must be modified and deplagarized.

This work presents a sum-of-squares (SOS) based framework to perform data-driven stabilization and robust control tasks on discrete-time linear systems where the full-state observations are corrupted by L-infinity bounded input, measurement, and process noise (error in variable setting). Certificates of state-feedback superstability, quadratic stability or positive stability of all plants in a consistency set are provided by solving a feasibility program formed by polynomial nonnegativity constraints. Under mild compactness and data-collection assumptions, SOS tightenings in rising degree will converge to recover the true superstabilizing or positive stabilizing controller, with some conservatism introduced for quadratic stabilizability. The performance of this SOS method is improved through the application of a theorem of alternatives while retaining tightness, in which the unknown noise variables are eliminated from the consistency set description. This SOS feasibility method is extended to provide worst-case-optimal robust controllers under H2  control costs. The consistency set description may be broadened to include cases where the data and process are affected by a combination of L-infinity bounded  measurement, process, and input noise.

\end{abstract}

%% file: sections/introduction.tex
\section{Introduction}
\label{sec:introduction}
\ac{DDC} is a group of methods that sidestep system-identification  in order to design controllers that regulate all possible plants that are consistent with observed data. 
The dynamical model considered in this paper is a discrete-time linear system with states $x_t \in \R^n$ and inputs $u_t \in \R^m$, for which measured data up to a finite time horizon of $T$ is available as $\dc = \{\hat{u}_t,\hat{x}_t\}_{t=1}^T$. The system includes $L_\infty$-bounded full-state measurement noise $\dx_t$, input noise $\du_t$, and process noise $w_t$ to form the model
\begin{subequations}
\label{eq:model}
\begin{align}
    x_{t+1} &= A x_t + B u_t + w_t \label{eq:discrete_dynamics}\\
    \hat{x}_t &= x_t + \dx_t,
    \quad   \hat{u}_t = u_t + \Delta u_t.
    \label{eq:noise_corrupt}
\end{align}
\end{subequations}

It is desired to find a constant matrix $K \in \R^{m \times n}$ generating state-feedback law $u = K x$ such that $A+BK$ is superstable \cite{polyak2001optimal, polyak2002superstable}, quadratically stable \cite{barmish1985necessary}, or positive stable \cite{farina2000positive} for all plants $(A, B)$
% to perform stabilization or robust optimal control of all plants 
that could have generated the data in $\dc$. An additional task is to formulate a controller $K$ that minimizes the worst-case-$H_2$-norm across all possible $\dc$-consistent systems $(A, B)$.
Expressing the model \eqref{eq:model} purely in terms of observations $\dc$ and noise processes $(\dx, \du, w)$,
\begin{equation}
      \label{eq:noise_bilinear}  \hat{x}_{t+1} - \dx_{t+1} = A (\hat{x}_t-\dx_t) + B(u_t-\Delta u_t) - w_t.
      \end{equation}
Equation \eqref{eq:noise_bilinear} involves multiplications $(A \dx_t, B \du_t)$ between unknown variables inside the description of the set $(A, B)$ of data-consistent plants. This bilinearity significantly increases the complexity of finding controllers $K$, since bilinearities typically yield NP-hard problems. This work formulates control of all consistent plants as a \ac{POP}, which is approximated by a converging sequence of \acp{SDP} through \ac{SOS} methods \cite{lasserre2009moments}. A theorem of alternatives based on robust \acp{SDP} is used to reduce the complexity of the generated \ac{SOS} programs by eliminating the noise variables $(\dx, \du, w)$.
      
Most prior work on \ac{DDC} involves process noise $w$ alone, and sets $\dx,\du=0$. One such method includes the work in \cite{de2019formulas}, which utilizes Willem's fundamental lemma \cite{willems2005note} to generate control policies based solely on the input-state-output data $\dc$. The method in \cite{de2019formulas} presents a set of data-driven programs that include stabilization and LQR control. Regularization methods have been applied to use Willem's fundamental lemma in the robust setting \cite{de2019formulas, coulson2019data, coulson2022robust, shafai2022data}, but bilinear dependence and fragility degrades performance in the derived controllers.

% Though easy to understand and implement, this approach is conservative in the sense that the robustness is not fully addressed. 

Other \ac{SDP}-based methods for nonconservative \ac{DDC} under $L_2$-bounded process noise (with $\dx,\du=0$) includes an S-Lemma \cite{van2020noisy}, Petersen's Lemma \cite{bisoffi2021data}, through updating uncertanties \cite{berberich2020combining}, and Lypaunov-Metzler (bilinear) inequalities for switched systems \cite{bianchi2022data}. 
The work in \cite{van2020informativity} defines a notion of `data informativity,' demonstrating that the assumptions required for a data-driven stabilization task is less restrictive than what is needed to perform system identification. 

The $L_\infty$ noise bound arises from error propagation of finite-difference approximations when computing derivatives and sampling. Another advantage of $L_\infty$ noise as compared to $L_2$ noise is that multiple datasets $\dc$ with differing time horizons can be concatenated without scaling or shifting the noise effects.
The work in \cite{berberich2020combining} briefly mentions adaptation for the $L_\infty$-bounded process noise case, while the computational complexity of $L_\infty$-bounded stabilization increases in an exponential manner with the number of measurements in $\dc \ (T)$. A SOS-based approach addressing the same problem can be found in \cite{dai2020data} which demands less complexity. Another way to reduce the complexity is to use the conservative notions of superstability \cite{polyak2001optimal}. 
% or quadratic stability \cite{barmish1985necessary} can yield more tractable programs. 
Further work on $L_\infty$-bounded process noise for superstabilization may be found in \cite{cheng2015,dai2018data,dai2022convex}. The work in \cite{miller2023ddcpos} uses \acp{LP} to perform positive stabilization of linear systems under $L_\infty$-bounded noise.
% Other methods for $L_\infty$-bounded control 

The case where measurement noise $\dx\neq 0$ is present is also called the \ac{EIV} setting. Prior work for \ac{EIV} has mostly concentrated on observation and system identification \cite{norton1987identification, cerone1993feasible, cerone2011set, soderstrom2018errors}. It is worth noting that in the $L_\infty$-bounded setting, the set of plants $(A, B)$ consistent with pure process noise $w$ forms a polytope, while the plants consistent with \ac{EIV} are generically contained in a non-convex region \cite{cerone1993parameter}. 

%% file: sections/contributions_arxiv.tex
The contributions of this paper are,
\begin{itemize}
    \item Formulation of stabilization under \ac{EIV} as a polynomial optimization problem
    \item Application of \ac{SOS} methods to recover (superstabilizing,  quadratically stabilizing, positive stabilizing) constant state-feedback stabilizing controllers of all consistent plants with recorded data
    \item Simplification of \ac{SOS} programs by using a Theorem of Alternatives to eliminate affine-dependent noise variables
    \item Analysis of computational complexity of \ac{SOS} programs
    \item Worst-case optimal $H_2$ control
    \item Proofs of continuity, polynomial approximability, and convergence
\end{itemize}

% \urg{
% % Sections of this work were accepted for presentation at the 61st Conference on Decision and Control \cite{miller2022eiv_short}. New content in this journal version as compared to the conference paper includes,
% \begin{itemize}
%     \item Quadratic stabilization (the previous version only included superstability)
%     \item Noting that a certificate function for superstability can be $\dx$-independent
%      \item Derivation and application of the matrix Theorem of Alternatives ensuring Positive Definiteness
%     \item Proofs of continuity, polynomial approximability, and convergence
%      \item Worst-case-$H_2$ optimal control of consistent plants
%     \item Further detail about the combination of measurement, input, and process noise
%     \item New extensions such as non-uniform sampling time.
%  \end{itemize}}

% \old{This paper proposes a convex, computationally tractable convex relaxation for  robust data driven control with  $\ell_\infty$ bounded measurement and process noise. Its main contributions are
%  \begin{enumerate}
%  \item To show that, in this scenario, robust superstabilizing  controllers can be designed by solving \ac{SOS}-based feasibility problems, which can be posed as \acp{SDP}. Robust stabilization is guaranteed by ensuring that all closed loop plants  consistent with the observed data are superstable. \item A theorem of alternatives reformulation that drastically reduces the number of variables involved and yields more tractable \acp{SDP} \cite{ben2015deriving}.
%  \end{enumerate}}

%% file: sections/summary_arxiv.tex
This paper is laid out as follows:
Section \ref{sec:preliminaries} introduces preliminaries such as acronym definitions, notation,  stability conditions for classes of linear systems, and \ac{SOS} methods. Section \ref{sec:full_method} creates a \ac{BSA} description of the consistency set of plants compatible with measurement-noise-corrupted data, and formulates \ac{SOS} algorithms to recover (super or quadratic) stabilizing controllers. Section \ref{sec:altern} reduces the computational complexity of these \ac{SOS} programs by eliminating the affine-dependent measurement noise variables through a Theorem of Alternatives. 
Section \ref{sec:optimal_control} applies the Full and Alternative methods to synthesize worst-case-optimal controllers for systems under measurement noise. 
Section \ref{sec:complexity} quantifies this reduction in computational complexity by analyzing the size and multiplicities of \ac{PSD} matrices involved in these \ac{SOS} methods.
Section \ref{sec:all_noise} extends the previously presents \ac{SOS} formulations to problems with measurement, input, and process noise. 
Section \ref{sec:examples} demonstrates the \ac{SOS} stabilizing algorithms on a set of examples. 
Section \ref{sec:extensions} extends the \ac{SOS} framework to cases including varying noise sets, non-uniform sampling times, and switched systems stabilization (with a known switching sequence). 
Section \ref{sec:conclusion} concludes the paper.
Appendix \ref{app:sos} reviews details of \ac{SOS} methods with scalar nonnegativity and matrix positive-semidefiniteness constraints.  
Appendix \ref{app:continuity} proves that multiplier functions for the Alternatives program may be chosen to be continuous. Appendix \ref{app:polynomial_sw} builds on this result and proves that the multiplier functions may also be chosen to be symmetric-matrix-valued polynomials.
% \old{
% % Section \ref{sec:sparse_altern} formulates a sparse tightening to the Alternatives program with reduced computational cost. 
% Section \ref{sec:examples} presents numerical experiments validating this method. Section \ref{sec:extensions} details extensions such as the varying noise sets, input noise, and the combination of process and measurement noise.
% % \urg{could cut down or swap out any of these. Non-uniform sampling time and switched systems (with tracking) are both available and have not been implemented yet.}
% The paper is concluded in Section \ref{sec:conclusion}. 
% }
% An extended version of this paper is available at \urg{Arxiv Link} and includes \urg{[content]}.

%% file: sections/preliminaries.tex
\section{Preliminaries}
\label{sec:preliminaries}
\subsection{Acronyms/Initialisms}
\input{sections/acronym}

\subsection{Notation}

The set of real numbers is $\R$, its $n$-dimensional vector space is $\R^n$, and its $n$-dimensional nonnegative real orthant is $\R^n_+$. The set of natural numbers is $\N$, and the subset of natural numbers between $1$ and $N$ is $1..N$. 
% The imaginary unit is $\mathbf{j} = \sqrt{-1}$.

The set of $m\times n$ matrices with real entries is $\R^{m \times n}$. The transpose of a matrix $Q$ is $Q^T$, and the subset of $n\times n$ symmetric matrices satisfying $Q^T = Q$ is $\psd^n$. The square identity matrix is $I_n \in \psd^n$. The rectangular identity matrix $I_{n \times m}$ is a  matrix whose main diagonal has values of 1 with all other entries equal to zero (consistent with MATLAB's $\textrm{eye}(n, m)$ function). The inverse of a matrix $Q \in \R^{n \times n}$ is $Q^{-1}$, and the inverse of its matrix transpose is $Q^{-T}$. The trace of a matrix $Q$ is $\Tr{Q}$. The Kronecker product of two matrices $A$ and $B$ is $A \otimes B$. The set of real symmetric \ac{PSD} matrices $\psd_+^n$ have all nonnegative eigenvalues $(Q \succeq 0)$, and its subset of \ac{PD} matrices $\psd^n_{++}$ have all positive eigenvalues $(Q \succ 0)$.  The $L_\infty$
operator norm of a matrix $M \in \R^{m \times n}$ is $\norm{M}_\infty = {\max_i \sum_{j=1}^n}\abs{M_{ij}}$. The asterisk operator $\ast$ may be used to fill in transposed entries of a symmetric matrix. The minimum and maximum eigenvalues of a matrix $Q \in \psd^n$ are $\lambda_{\min}(Q)$ and $\lambda_{\max}(Q)$ respectively. The elementwise division between vectors $a, b \in \R^n$ is $a./b$.

The set of polynomials in variable $x$ with real coefficients is $\R[x]$. The degree of a polynomial $p(x) \in \R[x]$ is $\deg p$. The set of polynomials with degree at most $d$ for $d \in \N$ is $\R[x]_{\leq d}$. The set of vector-valued polynomials is $(\R[x])^n$ and the set of matrix-valued polynomials is $(\R[x])^{m \times n}$. The subset of $n \times n $ symmetric-matrix-valued polynomials is $\psd^n[x]$, and its subcone of \ac{PSD} (PD) polynomial matrices is $\psd^n_+[x] \ (\psd^n_{++}[x])$. 
% An $m\times n$ array of $q \times q$ symmetric-matrix-valued polynomials is $(\psd^s[x])^{m \times n}$.
The set of \ac{SOS} polynomials is $\Sigma[x]$,
% , and a $n$-dimensional vector of \ac{SOS} polynomials is $(\Sigma[x])^n \subset (\R[x])^n$
and the set of \ac{SOS} matrices of size $n\times n$ is $\Sigma^n[x] \subset  \psd^n[x]$. The set of \ac{WSOS} polynomials over \iac{BSA} set $\K$ is $\Sigma[\K]$, with \ac{WSOS} matrices over the same set denoted as $\Sigma^n[\K]$.

The projection operator $\pi^x: (x, y) \mapsto x$ applied to a set $X \times Y$ is $\pi^x(X \times Y) = \{x \mid (x, y) \in X \times Y\}$.

% \urg{CUT (Mario's preference): An $m\times n$ array of \ac{SOS} matrix polynomials is notated similarly as $(\Sigma^s[x])^{m \times n}$.}

\subsection{Stability of Discrete-Time Linear Systems}
\label{sec:stability}
Let $x_{t+1} = A x_t + B u_t$ be a discrete-time linear system.
A system $A^{cl} = A+BK$ under the control law $u_t = K x_t$ is \textit{Schur} (stable) if all eigenvalues of $A+BK$ have absolute values less than 1. This subsection will provide certificates for the more conservative but computationally tractable notions of superstability, quadratic stability, and positive-stability.

% \textcolor{red}{For discrete system, we should use Schur instead of Hurwitz. I also don't see why we mention Hurwitz here. If I remember correctly, QS and Schur/Hurwitz stability are equivalent for linear system. See for instance "Routh-Hurwitz Stability Criterion and
% its Equivalents". }

\subsubsection{(Extended) Superstability}

The closed-loop system $A^{cl}=A+BK$ is superstable \cite{polyak2001optimal, polyak2002superstable} if
\begin{equation}
    \norm{A + BK}_\infty < 1 \quad \text{($L_\infty$ Operator Norm).} \label{eq:superstability_norm}
\end{equation}

Consequences of superstability are that $\norm{x}_\infty$ is a polyhedral \ac{CLF}, and that each pole $a+b \mathbf{j}$ of $A^{cl}$ satisfies $\abs{a} + \abs{b} < 1$.
Letting $\gamma = \norm{A + BK}_\infty < 1$, a superstable system will satisfy $\norm{x_t}_\infty \leq \gamma^{t/n} \norm{x_0}_\infty$ for any initial condition $x_0$ to the closed loop system $A + BK$ \cite{SZNAIER19963550, polyak2002superstable}. A system is Extended-Superstable if there exists a vector $v>0$ and a matrix $Y = \diag{v}$ such that \cite{polyak2004extended}
\begin{equation}
    \norm{Y^{-1}(A + BK)Y}_\infty < 1 \quad. \label{eq:ext_superstability_norm}
\end{equation}
The weighted $L_\infty$ norm $\norm{x./v}_\infty$ is a Lyapunov function of any system satisfying \eqref{eq:ext_superstability_norm}.
Superstability in \eqref{eq:superstability_norm} is the  case where $v = \1_n$. A system $A+BK$ is extended superstable if there exists a matrix $M \in \R^{n \times n}$, a vector $v>0$ with $Y = \diag{v}$, and a matrix $S \in \R^{m \times n}$ such that \cite[Theorem 1]{polyak2004extended}:
\begin{subequations}
\label{eq:ext_superstable}
\begin{align}
    &\textstyle \sum_{j=1}^n M_{ij} < v_i & & \forall i \in 1..n \label{eq:superstable_strict} \\
    &\textstyle -M_{ij} \leq A_{ij}v_j + \sum_{k=1} B_{ik}S_{kj} \leq M_{ij} & &\forall i,j \in 1..n.\label{eq:superstable_nonstrict}
\end{align}
\end{subequations}
An extended superstabilizing controller may be chosen as $K = S Y^{-1}$ if \eqref{eq:ext_superstable} is feasible. The matrix $M$ may be selected as $M_{ij} = \abs{A^{cl}_{ij} Y}$ for each $i, j \in 1..n$.

% Extended Superstability in \eqref{eq:superstability_norm} may be realized by imposing an exponential $n(2^n)$ number of strict linear inequality constraints:
% \begin{align}
%     \textstyle \sum_{j=1}^n(-1)^s_j A^{cl}_{ij} & < 1 & \forall i=1..n, s\in\{0,1\}^n.
% \end{align}

% A more efficient method of imposing superstability is through applying convex lifts
% from \cite{yannakakis1991expressing}. A matrix $M \in \R^{n\times n}$ can be introduced satisfying
% \begin{subequations}
% \label{eq:superstable}
% \begin{align}
% &\textstyle \sum_{j=1}^n M_{ij} < 1 & & \forall i = 1..n \label{eq:superstable_strict} \\
%     &-M_{ij} \leq A^{cl}_{ij} \leq M_{ij} & &\forall i,j = 1..n.\label{eq:superstable_nonstrict}
% \end{align}
% \end{subequations}
% There are $2n^2$ linear inequality and $n$ strict linear \textcolor{black}{in}equality constraints in \eqref{eq:superstable}. If the closed loop matrix $A^{cl}$ is superstable, such an $M$ can be chosen as
% \begin{align}
%     M_{ij} &= \abs{A_{ij} + \textstyle \sum_{\ell=1}^m B_{i \ell}K_{\ell j}} &  \forall i,j = 1..n
% \end{align}
% Superstability may be lost under a change-of-basis transformation of $A^{cl}$. 
\subsubsection{Positive-Stabilization}

The system $x_{t+1} = A x_t + B u_t$ is \textit{internally positive} if the initial condition is positive ($x_0 \in \R^n_{>0}$ element-wise) and input sequence is nonnegative ($\forall t: u_t \in \R^n_{\geq 0}$), then $\forall t: x_t \in \R^n_{\geq 0}$ \cite{farina2000positive}. A discrete-time system is internally positive if $(A, B)$ all have nonnegative entries. An uncontrolled linear system $x_{t+1} = A x_t$ is positive if $x_0 > 0 \Longrightarrow \forall t: x_t \geq 0$, which will occur when $A$ is a nonnegative matrix.
A necessary and sufficient condition for a nonnegative matrix $A$ to be Schur is that there exists a vector $v \in \R_{>0}^n$ such that
\begin{align}
\label{eq:stab_dlclf}
    v - A v \in \R_{>0}^n.
\end{align}
The vector $v$ \eqref{eq:stab_dlclf} inspires \iac{DLCLF} $V(x) = \max(x./v)$. 

A system  $x_{t+1} = A x_t + B u_t$ is positive-stabilizable by a state-feedback law $u_t = Kx_t$ if there exists a vector $v \in \R_{>0}^n,$ a matrix $Y = \diag{v}$, and a matrix $S \in \R^{m \times n}$, such that \cite{rami2007controller}
\begin{subequations}
\label{eq:stab_clean_d}
\begin{align}
        &v-(A Y + B S) \1_n  \in \R^n_{>0}  \label{eq:stab_clean_d_geq}\\
        &AY+BS \in \R^{n \times n}_{\geq 0}. 
\end{align}
\end{subequations}
The controller matrix $K$ may be recovered by $K = S Y^{-1}$. Condition \eqref{eq:stab_clean_d} involves $n$ strict inequality constraints and $n^2$ nonstrict inequality constraints.

\subsubsection{Quadratic Stabilizability}
% Superstability in the sense of \cite{polyak2002superstable} may be an overly restrictive criterion in finding a controller $K$. 
Quadratic stabilizability is another notion of stability that could be used when synthesizing control laws. A discrete-time system of the form in \eqref{eq:discrete_dynamics} is quadratically stabilizable if there exists \iac{PD} matrix $Y \in \psd_{++}^n$ and a matrix $S \in \R^{m \times n}$ such that \cite{barmish1985necessary}
\begin{equation}
\label{eq:quad_lmi_def}
   P(A, B) =  \begin{bmatrix}
    Y & A Y + B S\\
    \ast & Y
  \end{bmatrix} \in \psd_{++}^{2n}.
\end{equation}

The controller $K$ may be recovered as $K = S Y^{-1}$. The matrix $P(A, B)$ is an affine function of $(A, B)$ for fixed parameters $(Y, S)$, and Equation \eqref{eq:quad_lmi_def} is therefore an \ac{LMI}. The function $x^T Y^{-1} x$ is a \ac{CLF}. 

\begin{rem}
 The set of positive-stabilizable systems is strictly contained in the set of extended superstable systems. There is not necessarily a containment between classes of systems rendered  extended superstable or quadratically stable by a single state feedback control $K$.
% There is not necessarily a containment between the sets of plants rendered superstable, quadratically stable, or positive stable by a single state feedback control $K$. Sign-constraints on the elements of $K$ can be imposed by pinning corresponding elements of $Y$ in a linear fashion (e.g. $Y_{12} \geq 0 \Longrightarrow K_{12} \geq 0$).
\end{rem}

\begin{rem}
    Sign-constraints on the elements of $K$ in the extended superstable and positive stable cases can  be imposed by pinning corresponding elements of $Y$ in a linear fashion (e.g. $Y_{12} \geq 0 \Longrightarrow K_{12} \geq 0$). This sign-setting task is difficult and nonconvex for quadratic stabilization.
\end{rem}

%% file: sections/acronym.tex
\begin{acronym}[DLCLF]
% \acro{DDC}{Data Driven Control}

\acro{BSA}{Basic Semialgebraic}

\acro{CLF}{Control Lyapunov Function}

\acro{DDC}{Data-Driven Control}

\acro{DLCLF}{Dual Linear Copositive Lyapunov Function}

\acro{EIV}{Error in Variables}
\acroindefinite{EIV}{an}{a}

% \acro{GAS}{Globally Asymptotically Stable}

\acro{LMI}{Linear Matrix Inequality}
\acroplural{LMI}[LMIs]{Linear Matrix Inequalities}
\acroindefinite{LMI}{an}{a}

% \acro{LQR}{Linear Quadratic Regulator}
% \acroplural{LMI}[LMIs]{Linear Matrix Inequalities}
% \acroindefinite{LQR}{an}{a}

\acro{LP}{Linear Program}
\acroindefinite{LP}{an}{a}
% \acro{OCP}{Optimal Control Problem}

% \acro{ODE}{Ordinary Differential Equation}

\acro{PMI}{Polynomial Matrix Inequality}
\acroplural{PMI}[PMIs]{Polynomial Matrix Inequalities}

\acro{POP}{Polynomial Optimization Problem}

\acro{PSD}{Positive Semidefinite}

\acro{PD}{Positive Definite}

\acro{SDP}{Semidefinite Program}
\acroindefinite{SDP}{an}{a}
% \acro{SIR}{Susceptible, Infected, Removed}

\acro{SOS}{Sum of Squares}
\acroindefinite{SOS}{an}{a}

\acro{WSOS}{Weighted Sum of Squares}

\end{acronym}

%% file: sections/full_method.tex
\section{Full Program}
\label{sec:full_method}

This section will present \ac{SOS} approaches towards recovering stabilizing controllers $K$ (according to the criteria laid out in Section \ref{sec:stability}) applicable for all plants consistent with data in $\dc$. In this section we set $\du=0, \ w=0$ to simplify explanation and notation while still preserving the $A \dx $ bilinearity. A detailed discussion with $\du, w \neq 0$ is provided in Section \ref{sec:all_noise}.

% This section will present an \ac{SOS} feasibility program to recover a superstabilizing controller $K$ compatible with all plants consistent with $\dc$.

% \urg{Program with SOS constraints in terms of $(A, B, \dx)$}

\subsection{Consistency Sets}

The \ac{BSA} consistency set $\bar{\pc}(A, B, \dx)$ of plants and noise values $(A, B, \dx) \in \R^{n \times n} \times \R^{n \times m} \times \R^{n \times T}$ that are consistent with data $\dc$ under a noise bound of $\epsilon$ is described by:
% (for ease of illustration, we assume that $\du$, $w$ = 0, a detailed discussion is provided in Section \ref{sec:all_noise}),
\begin{align}
    \bar{\pc}: \ \begin{Bmatrix*}[l]0=- \dx_{t+1} + A\dx_{t} + h_t^0  & \forall t = 1..T-1 \\
     \norm{\dx_t}_\infty \leq \epsilon & \forall t = 1..  T\end{Bmatrix*}, \label{eq:pcbar}
\end{align}
with an intermediate definition of the affine weights $h^0_t$ as
        \begin{align}
            h_t^0 &= \hat{x}_{t+1} - A \hat{x}_t - B u_t & \forall t = 1..T-1. \label{eq:aff_weight}
        \end{align}
        % C(\mathcal{D}) &= \pi^{A,B} \bar{C}(\mathcal{D})

\begin{rem}
Multiple observations $\{\dc_{k}\}_{k=1}^{N_d}$ of the same system may be combined together by \ac{BSA} intersections to form $\bar{\pc} = \cap_{k=1}^{N_d} \bar{\pc}(\dc_k)$.
\end{rem}

The semialgebraic set of plants $\pc(A, B)$ consistent with the data in $\dc$ is the projection
\begin{align}
    \pc(A, B) = \pi^{A,B} \bar{\pc}(A, B, \dx). \label{eq:pc}
\end{align}

\begin{rem}
The consistency sets $\bar{\pc}$ and $\pc$ may be nonconvex and could even be disconnected.
\end{rem}

\begin{rem}
The describing constraints of $\bpc$ are bilinear in terms of the groups $(A, B)$ and $(\dx)$. Checking membership for fixed plant $(A_0, B_0) \in \pc$ may be accomplished by solving a feasibility \ac{LP} in terms of $\dx$.
\end{rem}

\begin{prob}
% \label{prob:super}
\label{prob:stable}
The data driven stabilization problem is to find $K$ such that $A^{cl} = A+BK$ is [Superstable, Extended Superstable, Positive Stable, Quadratically Stable] for all $(A, B) \in \pc.$
% \begin{align}
% \label{eq:prob_super}
%     \textrm{find}_K \ & A+BK \text{is [superstable & \forall (A, B) \in \pc.
% \end{align}
\end{prob}

\subsection{Function Programs}

This section will pose problem \ref{prob:stable} as a set of polynomial optimization programs, one for each class of stability. All programs will require the following assumption (for later convergence)
The following assumption is required for finite convergence of Problem \ref{prob:super_full}:
\begin{assum}
\label{assum:compact}
The sets $\bar{\pc}$ (and therefore $\pc$) are compact (Archimedean).
\end{assum}
Assumption \ref{assum:compact} may be satisfied if sufficient data is collected. 

% \urg{It looks like 'rmk', 'prob', 'assum' in autart.cls all share the same counter 'thm'. Is that standard in automatica papers? I don't think so. It depends on how you define your preamble.}

\subsubsection{Extended Superstability}

Extended superstabilization by a fixed $K \in \R^{m \times n}$ will be verified through equation \eqref{eq:ext_superstable} for all plants in $\pc$. The $M$ matrix in \eqref{eq:ext_superstable} will be a matrix-valued function $M(A, B): \R^{n \times n} \times \R^{n \times m} \rightarrow \R^{n \times n}$. The matrix function $M(A, B)$ will be $\dx$-independent given that the matrix $A+BK$ is also $\dx$-independent.

\begin{prob}
\label{prob:super_full}
 Superstabilizing Problem \ref{prob:stable} for a small margin $\delta > 0$ may be solved by
\begin{subequations}
\label{eq:super_full}
\begin{align}
    \find_{ v \in \R^{n}_{>0}, \ Y \in \R^{m \times n}} \ & \forall (A, B, \dx) \in \bar{\pc}(A, B, \dx):\\
    & \quad \forall i = 1..n: \label{eq:super_full_row}\\
    & \qquad v_i - \delta - \textstyle\sum_{j=1}^n M_{ij}(A, B) \geq 0 \nonumber  \\
    & \quad \forall i = 1..n, \ j = 1..n: \label{eq:super_full_element}\\
    & \qquad  M_{ij}(A, B) -\left(A_{ij}v_j + \textstyle \sum_{\ell=1}^m B_{i \ell}S_{\ell j}\right) \geq 0 \nonumber\\
    & \qquad  M_{ij}(A, B) +\left(A_{ij} v_j + \textstyle \sum_{\ell=1}^m B_{i \ell}S_{\ell j}\right) \geq 0.\nonumber
\end{align}
\end{subequations}
\end{prob}

%  The function $M$ can be taken to be a matrix valued polynomial that only depends on $(A,B,\dx)$.

% \begin{rem}
% \textcolor{black}{From \eqref{eq:super_full}, it follows that $M$ is infinite-dimensional. However, it is proved in the sequel that $M$ can be taken be a matrix valued polynomial that only depends on $(A,B,\dx)$.}
% \end{rem}

\begin{lem}
\label{lem:m_cont_select}
There exists a continuous selection for \\ 
$M(A, B)$ given $K$ under Assumption \ref{assum:compact}.
\end{lem}
\begin{proof}
% \urg{This is Jared's attempt to deplagarize, but it will be similar to the ARX work}
Define $\mathcal{M}: \pc \rightrightarrows \R_+^{n \times n}$ as the set-valued map (solution region to \eqref{eq:super_full_row}-\eqref{eq:super_full_element}):
\begin{equation}
\label{eq:m_map}
    \mathcal{M}(A, B): \begin{Bmatrix}\begin{array}{rc}
      \textstyle \forall i: & \sum_{j} M_{ij} \leq 1-\delta \\
        \forall (i,j): & -M_{ij} \leq \pm (A_{ij}v_j + \sum_{\ell=1}^m B_{i \ell} Y_{\ell j} )
      \end{array} \end{Bmatrix}.
\end{equation}
The right-hand sides of the constraints in \eqref{eq:m_map} are each continuous (linear) functions of $(A, B)$. This ensures that $\mathcal{M}$ is lower semi-continuous (Definition 1.4.2 in \cite{aubin2009set}) under the affine (continuous) changes in $(A, B)$ in the compact domain $\pc$ by  Theorem 2.4 in \cite{mangasarian1987lipschitz} (perturbations of right-hand-sides of linear-inequality-defined regions). Michael's theorem (Proposition 9.3.2 in \cite{aubin2009set}) suffices to show that a continuous selection of $M \in \mathcal{M}(A, B)$ exists, given that $\mathcal{M}$ takes on closed convex values in the Banach space $\R^{n \times n}$, has a compact domain, and is lower-semicontinuous. One such continuous selection is the Minimal Map $M(A, B) = \argmin_{\tilde{M} \in \mathcal{M(A, B)}} \norm{\tilde{M}}_F^2$.
\end{proof}

\begin{lem}
There exists a $\delta' > \delta$ such that a polynomial $M^p(A, B)$ may be chosen for $M(A, B)$ given $K$.
% The function $M(A, B, \dx, K)$ can be taken to be a polynomial $M^p(A, B, \dx)$.
\end{lem}
\begin{proof}
% \urg{Same deal, Jared's attempt}.
Let $\epsilon > 0$ be a tolerance such that $\forall (i, j): \sup_{(A, B) \in \pc} \abs{(M_{ij}(A, B) + \epsilon)- M_{ij}^p(A, B)} \leq \epsilon$ by the Stone-Weierstrass theorem in the compact set $\pc$ \cite{stone1948generalized}.  This implies that $M^p(A, B) \geq M(A, B)$ everywhere in $\pc$, because the residual $r_{ij}(A, B) = M^p_{ij}(A, B) - M_{ij}(A, B)$ takes on values between $[0, 2\epsilon]$.
% Define the residual $r_{ij}$ as $r(A, B) = M^p(A, B) - (M(A, B) \urg{+} \epsilon)$ with a range of $[-\epsilon, \epsilon]^{n \times n}$, with the implication that $M^p(A, B) \geq M(A, B)$ everywhere in $\pc$. The inequality $M_{ij} \pm A^{cl}_{ij} \geq 0$ from \eqref{eq:super_full_element} may be expressed as $M^p(A, B) - r(A, B) \urg{-} \epsilon \pm A^{cl}_{ij} \geq 0$, with an infimal value when $r(A, B) = -\epsilon$ of $M^p(A, B) \pm A^{cl}_{ij} \geq M(A, B) \pm A^{cl}_{ij} \geq 0$. 
Now consider \eqref{eq:super_full_row} with $M^p$:
\begin{subequations}
\begin{align}
1 - \delta - \textstyle \sum_{j=1}^n M_{ij}^p &= 1 - \delta - \textstyle \sum_{j=1}^n (M_{ij} + r_{ij}) \label{eq:row_r} \\
%     \intertext{The worst-case value of $r$ in \eqref{eq:row_r} is when $r_{ij}(A, B) = 2\epsilon$, leading to}
&\geq 1 - \delta - 2 \epsilon n - \textstyle \sum_{j=1}^n (M_{ij}) \label{eq:row_r2}.
\end{align}
\end{subequations}
For each row $i$, define $Z^*_i$ as
\begin{align}
    Z^*_i &= \sup_{(A, B) \in \pc} \textstyle \sum_{j=1}^n (M_{ij}) \leq 1-\delta,
    \end{align}
    from which it holds via \eqref{eq:super_full_row} that
    \begin{equation}
    \forall i =1..n: \quad 1-\delta - Z^*_i \geq 0 \quad \implies \quad 1 - \delta/2 - Z^*_i > 0. \label{eq:super_full_delta}
\end{equation}

Substituting \eqref{eq:super_full_delta} into \eqref{eq:row_r2} under the condition that \eqref{eq:row_r2} must be nonnegative yields
\begin{equation}
    (1-\delta/2 - Z^*_i) - 2\epsilon n \geq \delta/2 > 0.
\end{equation}

Choosing $\delta' = \delta/2 + 2n \epsilon$ with $\epsilon < (1-\delta/2)/(2n)$ (to ensure that $\delta' < 1$) will certify that $M^p$ satisfies all inequality constraints w.r.t. $\delta'$.
% Because $\lim_{\epsilon \rightarrow 0} \delta'=\ \delta$ and $\delta \in (0, 1)$, choosing a sufficiently small $\epsilon$ (and resultant high degree $M^p$) will result in a permissible $M^p$. 
% The result that $\delta' < \delta$ may be interpreted that the controlled plants in $\pc$ must be even more superstable than $\delta$ in order to restrict to polynomial choices $M^p$ for $M$.
\end{proof}

\subsubsection{Positive Stability}

Positive-stabilization will occur by finding a common \ac{DLCLF} $\max(x./v)$ with a constant vector $v \in \R^n_{>0}$.

\begin{prob}
\label{prob:pos_full}
A formulation for Positive Stabilization of Problem \ref{prob:stable} is
\begin{subequations}
\label{eq:pos_full}
\begin{align}
    \find_{v, S} \qquad & v \in \R^n_{>0}, \quad Y = \diag{v}, \  \quad S \in \R^{m \times n} \\
    & \forall (A, B, \dx) \in \bpc: \nonumber \\
    & \qquad v - (A Y + B S) \1 -\delta \in \R^n_{\geq 0} &  \label{eq:find_lyap}\\
    & \qquad A Y + B S  \in \R^{n\times n}_{\geq 0}.\label{eq:find_pos}
\end{align}
\end{subequations}
The controller is returned by $K = S Y^{-1}$.
\end{prob}

\subsubsection{Quadratic Stabilizability}

Quadratic stabilizability of every plant in $\pc$  according to \eqref{eq:quad_lmi_def} will be enforced (if possible) by a common $K = M Y^{-1}$.
% Quadratic stabilization by equation \eqref{eq:quad_lmi_def} will 

\begin{prob}
\label{prob:quad_full}
A program for Quadratic Stabilization of Problem \ref{prob:stable} is
\begin{align}
\label{eq:quad_full}
   \find_{Y, M} &  \begin{bmatrix}
    Y & A Y + B S \\
    \ast & Y
  \end{bmatrix} \succ 0, & \forall (A,B, \dx) \in \bpc \nonumber \\
  & Y \in \psd^n_{++}, \ S \in \R^{m \times n}.
\end{align}
\end{prob}

The function $x^T Y^{-1} x$ is a common \ac{CLF} for all plants in $\pc$. 

\subsection{SOS Program and Numerical Considerations}
Problems \ref{prob:super_full},  \ref{prob:pos_full} and \ref{prob:quad_full} may each be approximated by \ac{WSOS} polynomials.
\subsubsection{SOS Preliminaries}
We briefly review notation used in defining \ac{WSOS} constraints for the imposition of \acp{PMI} certifying that a symmetric-matrix-valued polynomial is \ac{PSD}. Further detail about these concepts (e.g. optimality bounds, \ac{SDP} complexity) is available in Appendix \ref{app:sos}.

\Iac{BSA} set is defined by a finite number of bounded degree polynomials $\{g_i(x)\}_{i=1}^{N_g}$ and $\{h_j(x)\}_{j=1}^{N_h}$:
\begin{equation}
\mathbb{K} = \{x\in \R^n \mid g_i(x) \geq 0, \ h_j(x) = 0\}. \label{eq:bsa}
\end{equation}

A matrix-valued polynomial $P(x) \in \psd^s[x]$ is \ac{PSD} ($P(x) \in \psd^s_+[x]$) if $\forall x \in \R^n: \ P(x) \in \psd_+^s$. A matrix-valued-polynomial is an \ac{SOS} matrix ($P(x) \in \Sigma^s[x]$) if there exists a vector of polynomials $v(x) \in \R[x]^q$ and a \textit{Gram} matrix $Q \in \psd_+^{q s}$ for some $q \in \N$ with $P(x) = (v(x) \otimes I_s)^T Q (v(x) \otimes I_s).$ The set $\Sigma^s[\mathbb{K}]$ of \ac{WSOS} matrices over \eqref{eq:bsa} is the class of matrices $P(x)$ such that there exists multipliers $\sigma_0(x) \in \Sigma^s[x], \ \sigma_i(x) \in \Sigma^s[x], \ \phi_j \in \psd^s[x]$ with 
\begin{equation}
P(x) = \sigma_0(x) + \textstyle \sum_i {\sigma_i(x)g_i(x)} + \textstyle \sum_j {\phi_j(x) h_j(x)}. \label{eq:psatz_noeps}    
\end{equation}
%  for some $\varepsilon$ (Scherer Psatz \cite{scherer2006matrix}). 
The set $\K$ satisfies an \textit{Archiemedean condition} if there exists an $R > 0$ such that  $R - \norm{x}_2^2 \in \Sigma^1[\K]$. Every \ac{PD}-valued matrix polynomial $P(x)$ over an Archimedean $\K$ satisfies $P(x) - \varepsilon I_q \succeq 0$ for some $\varepsilon > 0$ (Theorem 2 of \cite{scherer2006matrix}, Scherer Psatz). The set $\Sigma^s[\mathbb{K}]$ is a subset of the set of matrix-valued-polynomials in $x$ that are \ac{PSD} over $\K$.

\subsubsection{Superstability SOS}

\ac{SOS} methods may be used to approximate the superstability Program  of \eqref{eq:super_full} by requiring that $M(A, B) \in (\R[A, B, \dx])^{n \times n}$ is a polynomial matrix of degree $2d$.

Define $q^{\textrm{row}}_i(A, B, \dx; v, S) $ as the left hand side of \eqref{eq:super_full_row}, and let $q^\pm_{ij}(A, B, \dx; v, S) $  be the left hand side of each constraint in \eqref{eq:super_full_element}. An example constraint from \eqref{eq:super_full_element} at $(i, j)$ may be written as

    \begin{equation}
    \label{eq:full_cone_nonneg}
        q^+_{ij}(A, B, \dx; v, S) = M_{ij}(A, B, \dx) -\left(A_{ij}v_j + \textstyle \sum_{\ell=1}^m B_{i \ell}S_{\ell j}\right). \nonumber
    \end{equation} 

% Program \ref{eq:super_full_wsos} expresses the degree-$d$ \ac{WSOS} tightening of Problem \ref{prob:super_full}.

Equation \eqref{eq:super_full_wsos} expresses the degree-$2d$ \ac{WSOS} tightening of Problem \ref{prob:super_full}, returning a controller $K = S \diag{1./v}$ if feasible:

% \begin{algorithm}[h]
% \caption{Full Superstability Program\label{alg:full}}
% \begin{algorithmic}
% % \KwIn{$d, \ \delta, \ \dc, \ \epsilon$}
% % \KwOut{ $K, \ M$ (or Infeasibility)}
% % $K \in \R^{n \times m}, \ M \in (\R[,B,\dx])^{n \times n}_{\leq 2d}$ \
% \Procedure{SS Full}{$d, \ \delta, \ \dc, \ \epsilon$}
% \State  Solve (or find infeasibility certificate):
% \State 
\begin{subequations}
\label{eq:super_full_wsos}
 \begin{align}
    \find_{v, S, M} \ & v \in \R^n_{>0}, \ S \in \R^{n \times m} \\
    & M \in \R[A,B]_{\leq 2d} \\
    & q^{row}_i \in \Sigma[\bpc]_{\leq 2d} & & \forall i \in 1..n \label{eq:super_full_wsos_put_row}\\
    & q^{\pm}_{ij} \in \Sigma[\bpc]_{\leq 2d} & & \forall i,j \in 1..n. \label{eq:super_full_wsos_put_pm}
\end{align}
\end{subequations}
%     \Return $K$
% \EndProcedure
% \end{algorithmic}
% \end{algorithm}

There are $2n^2+n$ nonnegativity constraints in Program \eqref{eq:super_full_wsos}, each requiring a degree-$2d$ \ac{WSOS} Psatz of \eqref{eq:putinar}. Each Psatz involves 
$n(n+m+T)$ variables $(A, B, \dx)$, which induces a Gram matrices of maximal size $\binom{n(n+m+T) + d}{d}$ at degree $d$.

\begin{thm}
\label{thm:ss_full_converge}
When all sets are Archimdean (assumption \ref{assum:compact}),
Program \eqref{eq:super_full_wsos} will recover a superstabilizing $K$ (if possible) solving as $d\rightarrow \infty$.
\end{thm}
\begin{proof}
This theorem follows from results on convergence of \acp{POP}. Applying the \ac{SOS} hierarchy to a \ac{POP} $p^* = \min_{x \in \K} p(x)$ for $\K$ Archimedean will result in a convergent sequence of lower bounds $p^*_d \leq p^*_{d+1} \ldots$ with $\lim_{d \rightarrow \infty}p_d^* = p^*$. The \ac{POP} realization of nonnegativity program \eqref{eq:super_full} is a feasibility problem with $p(x) = p^* = 0$. The lower bounds of the \ac{SOS} hierarchy will therefore take a value of $p_d^* = 0$ (feasible, superstabilizing $K$ found) or $p^*_d = -\infty$ (dual infeasible, superstabilizing $K$ not found). By the limit property of $\lim_{d \rightarrow \infty}p_d^* = p^*$ under the Archimedean assumption  \ref{assum:compact}, a superstabilizing $K$ will be found if possible as the degree $d$ increases. 
\end{proof}

\subsubsection{Positive SOS}

The positive stabilization \ac{SOS} program at degree $d$ with respect to a tolerance $\delta>0$ is
\begin{subequations}
\label{eq:pos_sos}
\begin{align}
    \find_{v, S} \qquad & v \in \R^n_{>0},  \quad S \in \R^{m \times n} \\
    &\textstyle  v_i  - \left(\sum_{j=1}^n\sum_{\ell=1}^m B_{i\ell} S_
    {\ell j}\right) \nonumber \\
    &\textstyle \qquad - v_i \left(\sum_{j=1}^n A_{ij} \right) -\delta \in \Sigma[\bpc]_{\leq 2d} & & \forall i \in 1..n \label{eq:find_lyap_sos}\\
    & \textstyle v_i A_{ij} + \sum_{\ell=1}^m B_{i\ell} S_{\ell j} \in \Sigma[\bpc]_{\leq 2d} & & \forall i, j \in 1..n.  \label{eq:find_pos_sos}
\end{align}
\end{subequations}
The controller is recovered by $K=S Y^{-1}$ if \eqref{eq:pos_sos} is feasible.
Program \eqref{eq:pos_sos} involves $n^2+n$ \ac{WSOS} Putinar Psatz
 constraints, each of degree $2d$. 
Just like in the Superstability case, the maximum Gram matrix size of \eqref{eq:pos_sos} is $\binom{n(n+m+T)+d}{d}.$
 
\begin{thm}
Program \eqref{eq:pos_sos} will recover a Positive stabilizing $K$ (if possible) as  $d\rightarrow \infty$ under the Archimedean assumption \ref{assum:compact}.
\end{thm}
\begin{proof}
This follows directly from the proof of Theorem \ref{thm:ss_full_converge}. All constraints in \eqref{eq:pos_sos} involve scalar polynomials, and do not require \acp{PMI} of size $2$ or greater.
\end{proof}
 
 % \begin{rem} \todo{OK for ArXiv, not journal}
% A more numerically stable way to recover $K$ (at least through MATLAB) is $K = (Y^{-1} S^T)^T$ by solving $m$ linear systems.
% \end{rem}

% \begin{rem}
% The superstability algorithm \ref{alg:full} involves a function variable $M(A, B, \dx)$ that changes with respect to $(A, B, \dx)$ while the controller $K$ is constant.
% In the quadratic stability case, the controller is derived from $K = S Y^{-1}$. Letting $(S, Y)$ depend on the variables yields a condition that $K = S(A, B, \dx) Y(A, B, \dx)^{-1}$ is constant, which implies that $\frac{\partial K}{\partial A}, \frac{\partial K}{\partial B}, \frac{\partial K}{\partial \dx} = \mathbf{0}$. These partial-derivative conditions are difficult and nonconvex to impose in terms of the coefficients of $(S(A, B, \dx), Y(A, B, \dx))$, so the quadratic stability program involves constant $(S, Y)$. \todo{This is by definition: quadratic stability implies a common Lyapunov function $Y$}
% % $(A, B, \dx)$
% % Enforcing that $K = S(A, B, \dx) Y(A, B, \dx)^{-1}$ is constant in $(A, B)$ given functions for $S, Y$ is difficult and nonconvex to impose.  \urg{ugly phrasing}
% \end{rem}

% \begin{rem}
% We note that the work in  \cite{lee2021lossless} offers an alternative formulation for quadratic stabilizability of a system \urg{Take out all of the Hurwitz Stabilizability sections and pack it into this remark. Also this means less repetition in the paper.}
% \end{rem}

\subsubsection{Quadratic SOS}

The quadratic stabilizability \ac{SOS} program is described in Equation \eqref{eq:quad_sos}, yielding a recovered controller $K = S Y^{-1}$ (if feasible):
% Algorithm \ref{alg:quad_full}.
% \begin{algorithm}[h]
% \caption{Full Quadratic Stabilizability Program \label{alg:quad_full}}
% \begin{algorithmic}
% \Procedure{Quad Full}{$d, \ \delta, \ \dc, \ \epsilon$}
% \State  Solve (or find infeasibility certificate):
% \State
\begin{align}
 \label{eq:quad_sos}
   \find_{Y, S} &  \begin{bmatrix}
    Y & A Y + B S \\
    \ast & Y
  \end{bmatrix} \in \Sigma^{2n}[\bpc]_{\leq 2d} \nonumber \\
  & Y \in \psd^n_{++}, \ S \in \R^{m \times n}.
\end{align}
%     \Return $K = S Y^{-1}$
% \EndProcedure
% \end{algorithmic}
% \end{algorithm}

Problem \eqref{eq:quad_sos} involves a \ac{PMI} constraint for a matrix of size $2n$ and a \ac{PD} constraint of size $n$. The \ac{PMI} contains $n(n+m+T)$ variables $(A, B, \dx)$. 
The maximal Gram matrix size as induced by  the degree-$d$ Scherer Psatz \eqref{eq:scherer_variables} is $2n \binom{n(n+m+T)+d}{d}$.

\begin{thm}
Program \eqref{eq:quad_sos} will recover a Quadratically stabilizing $K$ (if one exists) as  $d\rightarrow \infty$ under the Archimedean assumption \ref{assum:compact}.
\end{thm}
\begin{proof}
This follows from the proof of Theorem \ref{thm:ss_full_converge} as modified for the Scherer Psatz in \cite{scherer2006matrix}.
\end{proof}

% \subsection{Hurwitz SOS}

%% file: sections/alternatives_method.tex
 
\section{Alternatives Program}
\label{sec:altern}
% \urg{Rewrite this section in terms of the Matrix case.}

This section will formulate and use a Matrix Theorem of Alternatives in order to reduce the computational expense of running Algorithms \eqref{eq:super_full_wsos}, \eqref{eq:pos_sos}, and \eqref{eq:quad_sos}. The cost savings are derived from elimination of the affine-entering noise variables $\dx$. 
% In the superstability case, this section will additionally impose that the certificate $M$ is solely a function of $(A, B)$ and is constant in $\dx$.
% \todo{you need to justify this}. 
% This imposition is necessary to successfully apply the Theorem of Alternatives and attain a decomposed and more tractable SOS program. A more general case of $\dx$-affine functions $M$ are explored in Section \ref{sec:extensions_affine}
\subsection{Theorem of Alternatives}
Let $q: \R^{n\times n} \times \R^{n \times m} \rightarrow \psd^s$ be a symmetric-matrix valued function satisfying the constraint
%changing everything to pd
\begin{align}
    q(A, B) &\in \psd_{++}^s & \forall (A, B) \in \pc.
    \label{eq:altern_nonneg}
\end{align}

If constraint \eqref{eq:altern_nonneg} is satisfied (feasible), then the following problem is infeasible:
% \todo{subtle point here. This is a necessary but not sufficient condition. In principle both can be infeasible unless they are strong alternatives.}
\begin{align}
    \find_{(A, B) \in \pc} & -\lambda_{\min}(q(A,B)) \geq 0. \label{eq:altern_eig}
\end{align}

\begin{lem}
\label{lem:strong_altern}
Constraints \eqref{eq:altern_nonneg} and \eqref{eq:altern_eig} are strong alternatives (either one or the other is feasible).
\end{lem}
\begin{proof}
If \eqref{eq:altern_nonneg} is feasible, then $\lambda_{\min} q(A, B)$ is positive for all $(A, B) \in \pc$. Therefore, there does not exist an $(A, B)$ such that $\lambda_{\min}(A, B) \leq 0$, which is the statement of \eqref{eq:altern_eig}. On the opposite side, feasibility of \eqref{eq:altern_eig} with an $(A', B'): \ \lambda_{\min}(q(A, B)) \leq 0$ implies that $q(A, B) \not\in \psd_{++}^s$, and therefore \eqref{eq:altern_nonneg} is infeasible. Additionally, there is no case where \eqref{eq:altern_nonneg} and \eqref{eq:altern_eig} are both infeasible: either all $q(A, B)$ are \ac{PD} \eqref{eq:altern_nonneg} or there exists a non-\ac{PD} counterexample \eqref{eq:altern_eig}.
\end{proof}

%  The feasibility of \eqref{eq:altern_eig} equivalently states that there exists an $(A, B, \dx) \in \bpc$ such that $q(A, B)$ has a nonpositive eigenvalue. Problems \eqref{eq:altern_nonneg} and \eqref{eq:altern_eig} are a pair of weak alternatives: at most one of them is feasible given a $q(A, B)$.

%  (where $\beta$ is the associated negative eigenvector). 
 A set of dual variable functions $(\mu(A, B), \zeta(A, B))$ 
%  $\mu_{ti}:   \pc \rightarrow \psd^s, \ \forall t=1..T-1, i=1..n$ and $\zeta^\pm_{ti}: \pc \rightarrow \psd^s_+, \ \forall t=1..T, i=1..n$ 
 may be defined based on the constraint description from \eqref{eq:pcbar}:
 \begin{subequations}
 \label{eq:mu_zeta_def}
 \begin{align}
     \mu_{ti}:&   \pc \rightarrow \psd^s & & \forall i=1..n, \  t=1..T-1 \\
     \zeta_{ti}^\pm: &   \pc \rightarrow \psd^s_+ & & \forall i=1..n, \ t=1..T.
 \end{align}
 \end{subequations}
 The multipliers from \eqref{eq:mu_zeta_def}, the Scherer Psatz \eqref{eq:scherer}, and the Robust Counterpart method of \cite{ben2015deriving} can  together be used to form the weighted sum $\Phi(A, B; \zeta^\pm, \mu): \pc \rightarrow \psd^s$ with

% \begin{equation}
% \label{eq:phi_apply}
    % \Phi = -P + \sum_{i,t} \left(\zeta_{ti}^+(\epsilon - \dx_t) +\zeta^-_{ti}(\epsilon + \dx_t)\right) + 
% \end{equation}
\begin{align}
    \Phi =& -q(A, B) + \textstyle \sum_{t=1,i=1}^{T, n} \left(\zeta_{ti}^+(\epsilon - \dx_t) +\zeta^-_{ti}(\epsilon + \dx_{t,i})\right) \nonumber \\
    &+\textstyle\sum_{t=1,i=1}^{T-1,n} \mu_{ti}(-\dx_{t+1, i} + A_i \dx_{ti} + h_{ti}^0). \label{eq:Sij}
\end{align}

The dual multipliers will always be treated as (possibly nonunique and discontinuous) functions $\mu(A, B)$ or $\zeta(A, B)$, but their $(A, B)$ dependence may be omitted to  condense notation. 

\begin{thm}
\label{thm:sup_weak}
% Given multipliers $\zeta^\pm, \mu$, the following expression is a sufficient certificate of infeasibility of program \eqref{eq:altern_eig} (where $\lambda_{\rm{max}}$ denotes the maximal eigenvalue)
A sufficient condition for infeasibility of program \eqref{eq:altern_eig} is if there exists multipliers $\zeta^\pm, \mu$ according to \eqref{eq:mu_zeta_def} such that
\begin{align}
    \forall(A,B)\in \pc: \ \sup_{\dx \in \R^{n \times T}} \lambda_{\rm{max}}(\Phi(A,B; \zeta^\pm, \mu)) < 0. \label{eq:altern_sup}
\end{align}
\end{thm}
\begin{proof}
This theorem holds by arguments from \cite{ben2015deriving, boyd2004convex} with modifications for the matrix case.

Any point $(A, B, \dx) \in \bpc$ must satisfy $\norm{\dx_t}_\infty \leq \epsilon$ for all times $t=1..T$ \eqref{eq:pcbar}. This implies that $\epsilon \pm  \dx_t$
are nonnegative vectors for each time $t$ for $\dx \in \pi^{\dx} \bpc$, and therefore $\zeta^\pm_{ti}(\epsilon \pm \dx_{ti})$ are each \ac{PSD} matrices given that $\zeta^\pm_{ti}$ are \ac{PSD}. Additionally, the data consistency constraints in $\dx_{t+1} + A\dx_{t} + h_t^0$ from \eqref{eq:pcbar} evaluate to 0 for $(A, B, \dx) \in \bpc$, so $\mu$ times these zero quantities result in a zero matrix. The addition of these \ac{PSD} and Zero multiplier terms to $-q$ ensures that $\lambda_{\max}(\Phi) \geq \lambda_{\max}(-q)$. Finding multipliers $(\zeta^\pm, \mu)$ such that \eqref{eq:altern_sup} holds therefore implies that $-q$ is Negative Definite ($q$ is \ac{PD}) over the space $\pc$. This definiteness statement certifies infeasibility of \eqref{eq:altern_eig}, because there cannot exist a negative eigenvalue of $q$ as $(A, B)$ ranges over $\pc$.
% The set $\pi^{\dx} \bpc$ is a subset of the constraints $\norm{\dx_t}$

% Feasibility of \eqref{eq:altern_eig} requires the existence of a nonnegative eigenvalue of $-q(A, B)$ for some $(A, B) \in \pc$, certifying that $q(A,B)$ is not \ac{PD}. The expression $\Phi$ in \eqref{eq:Sij} is the summation of $-q$ with \ac{PSD} terms (multiplicands of $\zeta^\pm$) and zero terms (multiplicands of $\mu$) when evaluated on $\bpc$. If relation \eqref{eq:altern_sup} holds $\Phi: ((-q)$ plus \ac{PSD} weighting terms$)$ lacks a nonnegative eigenvalue, and therefore \eqref{eq:altern_eig} is infeasible. 
\end{proof}
\begin{thm}
\label{thm:sup_strong} Theorem \ref{thm:sup_weak} is additionally necessary for infeasibility of \eqref{eq:altern_eig} in the case where $s=1$ ($q(A, B)$ is scalar).
\end{thm}
\begin{proof}
The term $\lambda_{\min}(q)$ is replaced by $q$ in the scalar case of \eqref{eq:altern_eig}.
The constraints in \eqref{eq:pcbar} are affine in $\dx$, which means they are both convex and concave in the variable $\dx$. The term $q(A, B)$ is also independent of $\dx$. Concavity of the constraints in $\dx$ is enough to establish necessity and strong alternatives by convex duality (Section 5.8 of \cite{boyd2004convex}).
% If there exists a point in $\bpc$ where $\lambda_{\rm{min}}{-q(A,B)}$.
 \end{proof}
 
\begin{rem}
 Refer to \cite{ben2002interval} for an example where the Alternatives procedure \eqref{eq:altern_sup} with $s>1$ is sufficient but not necessary (robust \acp{SDP} over interval matrices: ``computationally tractable conservative approximation'').
The work in \cite{zhen2022robust} formulates robust \acp{SDP} involving polytopic uncertainty as a generally intractable two-stage optimization program.
\end{rem}
Equation \eqref{eq:phi_sum} can be simplified and transformed into a feasibility program by explicitly defining and constraining its $\dx$-supremal value.
The sum $\Phi$ is affine in $\dx$, and the constant terms (in $\dx$) of $\Phi$ are $Q(A, B; \zeta^\pm, \mu)$:
\begin{align}
\label{eq:altern_const}
    Q &= -q(A, B) + \textstyle \sum_{t=1,i=1}^{T, n} \epsilon \left(\zeta_{ti}^+ + \zeta^-_{ti}\right) +\textstyle\sum_{t=1,i=1}^{T-1,n} \mu_{ti}h^{0}_{ti}.
\end{align}

The sum $\Phi$ may therefore be expressed as
\begin{align}
\label{eq:phi_sum}
    \Phi =& Q + \textstyle\sum_{t=1,i=1}^{T-1,n} \mu_{ti}A_i \dx_{ti} - \textstyle\sum_{t=2,i=1}^{T,n} \mu_{t-1,i}\dx_{ti} \nonumber\\
    &+ \textstyle\sum_{t=1,i=1}^{T, n} (\zeta^-_{ti} - \zeta^+_{ti})\dx_{ti}.
\end{align}

The supremal value of $\Phi$ in \eqref{eq:altern_sup} given $(A, B; \zeta^\pm, \mu)$ is
\begin{equation}
    \label{eq:dual_function_quad}
    \sup_{\dx} \lambda_{max} (\Phi)  = \begin{cases}
     \lambda_{\rm{max}}(Q) &   \zeta_{1i}^+ - \zeta^-_{1i} = \textstyle \sum_{j=1}^n A_{ji}  \mu_{1j} \\
      & \zeta_{ti}^+ - \zeta^-_{ti}= \textstyle \sum_{j=1}^n A_{ji}  \mu_{tj}   -  \mu_{t-1,i}\\
      & \zeta_{Ti}^+- \zeta^-_{Ti} = - \mu_{T-1, i} \\
     \infty & \textrm{otherwise.}
    \end{cases}
\end{equation}

A feasibility program that ensures $\sup_{\dx} \lambda_{max} (\Phi(A,B)) < 0, \ \forall (A,B) \in \pc$ (ranging over $i=1..n$) is
% \todo{confusing notation:$\zeta,\mu$ are vectors but you are using the notation for symmetric matrices? Why are these symmetric and why do you have a PSD constraint on $\zeta$?}
\begin{subequations}
\label{eq:altern_feas}
\begin{align}
    \find_{\zeta, \mu} \ & Q(A, B; \zeta^\pm, \mu) \prec 0  & & \forall (A, B)\in \pc\label{eq:altern_feas_Q}\\
&\zeta_{1i}^+ - \zeta^-_{1i} = \textstyle \sum_{j=1}^n A_{ji}  \mu_{1j}  \label{eq:altern_feas_zeta1}\\
& \zeta_{ti}^+ - \zeta^-_{ti}= \textstyle \sum_{j=1}^n A_{ji}  \mu_{tj}   -  \mu_{t-1,i} & & \forall t \in 2.. T- 1 \\
      & \zeta_{Ti}^+ - \zeta^-_{Ti} = - \mu_{T-1, i} \label{eq:altern_feas_zetaT} \\
      & \zeta_{ti}^\pm(A,B) \in \psd_+^s & & \forall t=1..T, \label{eq:altern_feas_exists_zeta} \\
    & \mu_{ti}(A,B) \in \psd^s & & \forall t=1..T-1.\label{eq:altern_feas_exists}
\end{align}
\end{subequations}

\begin{thm}
\label{thm:feas_strong}
% Program \eqref{eq:altern_feas} is equivalent to the  statement in \eqref{eq:altern_nonneg}, and is a strong alternative to \eqref{eq:altern_eig}.
When $s=1$, program \eqref{eq:altern_feas} is equivalent to the  statement in \eqref{eq:altern_nonneg}, and is a strong alternative to \eqref{eq:altern_eig}. When $s > 1$, program \eqref{eq:altern_feas} is a weak alternative to \eqref{eq:altern_eig}.
\end{thm}
\begin{proof}
Lemma \ref{lem:strong_altern} proves that \eqref{eq:altern_nonneg} and \eqref{eq:altern_eig} are strong alternatives. 

In the case where $s=1$, Theorem \ref{thm:sup_strong} proves that \eqref{eq:altern_eig} and \eqref{eq:altern_sup} are strong alternatives. Given that \eqref{eq:altern_feas} is an explicit condition for validity of \eqref{eq:altern_sup}, it holds that \eqref{eq:altern_feas} and \eqref{eq:altern_eig} are strong alternatives and therefore \eqref{eq:altern_nonneg} and \eqref{eq:altern_feas} are equivalent.

In the more general case where $s>1$, Theorem \ref{thm:sup_weak} proves that \eqref{eq:altern_eig} and \eqref{eq:altern_sup} are weak alternatives. Successfully finding a certificate $(\zeta^\pm, \mu)$ from \eqref{eq:altern_feas} validates that \eqref{eq:altern_eig} is infeasible. It is not possible for both \eqref{eq:altern_feas} and \eqref{eq:altern_eig} to hold simultaneously.  However, there  may still exist cases where \eqref{eq:altern_nonneg} holds but \eqref{eq:altern_feas} is infeasible.
\end{proof}

% \begin{thm}
% \label{thm:dual_continuous}
% The dual multipliers $\zeta_{ti}^\pm(A, B)$ and $\mu_{ti}(A, B)$ which certify that $q(A, B) \succ 0$ over $\pc$ may be chosen to be polynomials.
% \end{thm}
% \begin{proof}
% See Appendices \ref{app:continuity} (continuous selections) and \ref{app:polynomial_sw} (polynomial approximation through Stone-Weierstrass).
% \end{proof}

\begin{thm}
\label{thm:dual_continuous}
The dual multipliers $\zeta_{ti}^\pm(A, B)$ and $\mu_{ti}(A, B)$ which certify that $q(A, B) \succ 0$ over $\pc$ via \eqref{eq:altern_feas} may be chosen to be continuous.
\end{thm}
\begin{proof}
See Appendix \ref{app:continuity}.
\end{proof}

\begin{thm}
\label{thm:dual_polynomial}
Whenever \eqref{eq:altern_feas} is feasible,  $\zeta_{ti}^\pm(A, B)$ and $\mu_{ti}(A, B)$ may additionally be required to be polynomial matrices.
\end{thm}
\begin{proof}
See Appendix \ref{app:polynomial_sw}.
\end{proof}

% \begin{rem}
% The dual multipliers $\zeta_{ti}^\pm(A, B)$ and $\mu_{ti}(A, B)$ may be chosen to have continuous selections by following the arguments of Theorem \ref{lem:m_cont_select}
% % \urg{Wrong. figure out selections and when they apply}
% \todo{Not sure if this is true. Solutions of LP are Lipschitz w.r.t changes in RHS, but here you are changing the matrix multiplying $\mu$ so Mangasarian's does not apply.} 
% \urg{can we do this? Might need to write it all out.}
% \end{rem}

\subsection{Alternatives SOS}

The degree-$2d$ \ac{WSOS} truncation of program \eqref{eq:altern_feas} to certify positive definiteness in \eqref{eq:altern_nonneg} is contained in Algorithm \ref{alg:altern_psatz}. The Alternatives Psatz in Algorithm \ref{alg:altern_psatz} requires the following assumption to ensure convergence as $d \rightarrow \infty$:
\begin{assum}
\label{assum:altern}
An Archimedean set $\Pi(A, B)\supseteq \pc$ is previously known.
\end{assum}
\begin{rem}
A set $\Pi$ may arise from prior knowledge about plant behavior and its reasonable limits (e.g. $(A,B)$ are contained in a known box).
\end{rem}

\begin{algorithm}[h]
\caption{Alternatives Psatz $(\Sigma_{\leq 2d}^{s, \alt}[\pc])$}\label{alg:altern_psatz}
\begin{algorithmic}
\Procedure{Altern Psatz}{$d, \ q(A, B), \Pi, \dc, \epsilon, s$}
\State Solve (or find infeasibility certificate):
\State \begin{subequations}
\label{eq:altern_sos}
\begin{align}
     &\zeta^\pm_{ti}(A, B) \in \Sigma^s[\Pi]_{\leq 2d} \label{eq:altern_sos_zeta} \\
    & \mu_{ti}(A, B) \in \psd^s[A, B]_{\leq 2d-1} \label{eq:altern_sos_mu}\\
    & -Q(A, B; \zeta^\pm, \mu) \in \Sigma[\Pi]_{\leq 2 d} \textrm{ (from \eqref{eq:altern_const})}\label{eq:altern_sos_Q}\\
      &\zeta_{1i}^+ - \zeta^-_{1i} = \textstyle \sum_{j=1}^n A_{ji}  \mu_{1j} \label{eq:altern_sos_1} \\
& \zeta_{ti}^+ - \zeta^-_{ti}= \textstyle \sum_{j=1}^n A_{ji}  \mu_{tj}   -  \mu_{t-1,i} \quad  \forall t \in 2.. T- 1 \label{eq:altern_sos_t}\\
      & \zeta_{Ti}^+ - \zeta^-_{Ti} = - \mu_{T-1, i} \label{eq:altern_sos_T}.
\end{align}
\end{subequations}
\Return $\zeta, \ \mu$ (or Infeasibility)
\EndProcedure
\end{algorithmic}
\end{algorithm}

\begin{rem}
All of the equality constraints in \eqref{eq:altern_sos_1}-\eqref{eq:altern_sos_T} are linear constraints in the coefficients of $\zeta^\pm, \mu$.
% \todo{Still do not understand the symmetry constraint in $\zeta,\mu$}
\end{rem}
\begin{rem}
Choosing $\zeta$ of degree $2d$ and $\mu$ of degree $2d-1$ ensures that the expression of $\Phi$ remains degree $2d$ when $\deg{q} \leq 2d$.
\end{rem}

% Algorithms \ref{alg:altern_ss} and \ref{alg:altern_qs} are the Alternatives \ac{WSOS} formulations of Programs

The Full programs in \eqref{eq:super_full_wsos}, \eqref{eq:pos_sos}, and \eqref{eq:quad_sos} may be converted to Alternatives programs by replacing the respective cone containments $\in \Sigma^s[\bpc]_{\leq 2d}$ by $\in\Sigma^{s, \alt}[\Pi]_{\leq 2d}$. The below program \eqref{eq:super_altern_wsos} is an example of this type of Alternatives conversion for the extended superstability program \eqref{eq:super_full_wsos} with the cone $\Sigma^{s, \alt}[\Pi]_{\leq 2d}$ rather than $\Sigma^{1}[\bpc]_{\leq 2d}$:
\begin{subequations}
\label{eq:super_altern_wsos}
\begin{align}
    v & \in \R^n_{>0}, \ S\in \R^{n \times m} \\
    M_{ij} &\in \R[A,B]_{\leq 2d} & & \forall i,j \in 1..n\\
    q^{row}_i &\in \Sigma_{\leq 2d}^{1,\alt}[\Pi] & & \forall i \in 1..n \label{eq:super_altern_wsos_put_row}\\
    q^{\pm}_{ij} &\in \Sigma_{\leq 2d}^{1,\alt}[\Pi] & & \forall i,j \in 1..n.\label{eq:super_altern_wsos_put_pm}
\end{align}
\end{subequations}

%% file: sections/optimal_control.tex
\section{Worst-Case-H2-Optimal Control}

\label{sec:optimal_control}
The previously presented Full and Alternative methods perform superstabilization or quadratic stabilization by attempting to find a single feasible controller $K$.
This section generalizes existing \acp{LMI} for discrete-time worst-case $H_2$-optimal  control over all possible linear systems in the consistency set $\pc$. 
The discrete-time linear system model used in this section has a state $x \in \R^n$, a control input $u \in R^m$, an exogenous input $w \in \R^e$, a matrix $E \in \R^{n \times e}$, and a regulated output $z \in \R^{r}$ defined by matrices $C \in \R^{r \times n}, D \in \R^{r \times m}$:  %, D_2 \in \R^{r\times e}
\begin{align}
    x_{t+1} &= Ax_t + Bu_t + E w_t \label{eq:control_sys}\\
    z_t &= Cx_t + Du_t . \nonumber
\end{align}
% \urg{Some sort of nondegeneracy must be posed over the matrices $(C, D)$, right?}
The matrices $(C, D, E)$ will be a-priori given, while a consistency set description $(A, B) \in \pc$ is available through the input-state collected data $\dc$. 
The original \acp{LMI} referenced in this section are compiled in \cite{caverly2019lmi} and utilize a convexification method from \cite{feron1992numerical}. We note that robust performance can also be addressed by minimizing the peak-to-peak gain using \acp{LP} for superstability \cite{polyak2001optimal} or positive-stability \cite{rami2007controller}. Minimizing peak-to-peak gain for extended superstability requires solving parametric \acp{LP} \cite[Section 6]{polyak2004extended}.

% \todo{You need way more details in Algorithms 4 and 5 to explain how to "alternativize" and get rid of $\Delta x$ in all cases. The paper has way too much details at the beginning as is way too terse in the new/important stuff}
% \subsection{H2-optimal control}

% The following \ac{LMI} to minimize the $H_2$-norm between $w \rightarrow z$ is expressed in Section 4.2.2 of \cite{caverly2019lmi}.
% \begin{prob}
% The minimal controlled $H_2$-norm of the system \eqref{eq:control_sys} for a single plant $(A_0, B_0)$ is upper bounded by $\nu^2$ for $\nu \in \R$ if the following \acp{LMI} hold,
% \begin{subequations}
% \label{eq:h2_single}
% \begin{align}
%     \find_{P, Z, F} \quad & \begin{bmatrix} P & A_0 P + B_0 F & E \\
%     \ast & P & \mathbf{0} \label{eq:h2_single_schur}\\
%     \ast & \ast & I \end{bmatrix} \in \psd_{++}^{2n + e} \\
%     & \begin{bmatrix} Z & C P + D F \\ \ast & P \end{bmatrix} \in \psd_{++}^{2n} \\
%     & \Tr Z < \nu \\
%     & P \in \psd^n_{++}, \ Z \in \psd^n, \ F \in \R^{m \times n}.
% \end{align}
% \end{subequations}
% The corresponding feedback gain is $K = F P^{-1}$.
% \end{prob} 

The following \ac{LMI} to minimize the $H_2$-norm between $w \rightarrow z$ is expressed in Section 4.2.2 of \cite{caverly2019lmi}:
% \urg{It is suggested to use the following instead of above, since it contains the smallest size of Gram matrix and directly corresponds to your algorithm. It is fine that $E$ is bilinear since it is assumed known. A few changes: a) I use $Y,M$ instead of $P,F$ to be consistent with other parts. b) replace $\nu$ by $\gamma$ to avoid confusion. c) size of $Z$ is r. d) modify the problem to be worst-case $H_2$ control.
\begin{prob}
\label{prob:h2_single}
There exists a static output feedback $u=Kx$ such that the $H_2$-norm of the system \eqref{eq:control_sys} is upper bounded by a value $\gamma \in \R^+$ if and only if the following \acp{LMI} are feasible:
\begin{subequations}
\label{eq:h2_single}
\begin{align}
    \find_{Y, Z, S} \quad & \begin{bmatrix} Y - EE^T & A Y + B S \\ \ast & Y \label{eq:h2_single_schur}\end{bmatrix} \in \psd_{++}^{2n} \\
    & \begin{bmatrix} Z & C Y + D S \\ \ast & Y \end{bmatrix} \in \psd_{++}^{n+r} \\
    & \Tr Z \leq \gamma^2 \\
    & Y \in \psd^n_{++}, \ Z \in \psd^r_+, \ S \in \R^{m \times n}.
\end{align}
\end{subequations}
The corresponding feedback gain is $K = SY^{-1}$.
\end{prob} 

Problem \eqref{eq:h2_single} may be posed as a worst-case $H_2$ problem by minimizing $\gamma^2$ and ensuring that the LMIs hold over all possible plants in the consistency set $\pc$.

\begin{prob}
There exists a static output feedback $u=Kx$ such that the worst-case optimal $H_2$-norm of the system \eqref{eq:control_sys} over a set $(A, B)\in \pc$ may be upper bounded by
\begin{align}
\inf_{\gamma \in \R^+, Y, Z, S} &  \qquad \gamma^2 \label{eq:h2_worst} \\
& \begin{bmatrix} Y - E E^T& A Y + B S \\
    \ast & Y \end{bmatrix} \in \psd_{++}^{2n}  \qquad \forall (A, B) \in \pc \nonumber \\
    & \begin{bmatrix} Z & C Y + D S \\ \ast & Y \end{bmatrix} \in \psd^{n+r}_+
    % \in \Sigma^{n+r,\alt}[\Pi]_{\leq 2d} 
    \nonumber \\
    & \gamma^2 - \Tr Z \geq 0
    % \in \Sigma^{1,\alt}[\Pi]_{\leq 2d}
    \nonumber \\
  & Y \in \psd^n_{++}, Z\in \psd^r_+, \ S \in \R^{m \times n}.  \nonumber
\end{align}
The feedback gain $K = SY^{-1}$ is constant in $(A, B)$.
\end{prob}

A degree-$2d$ Alternatives problem may be posed for worst-case $H_2$ synthesis by restricting the top \ac{LMI} in \eqref{eq:h2_single} (holding over $\forall (A,B) \in \pc$) to the cone $\Sigma^{2n,\alt}[\Pi]_{\leq 2d}$, as in
\begin{align}
\min_{\gamma \in \R^+, Y, Z, S} &  \qquad \gamma^2 \label{eq:h2_wsos} \\
& \begin{bmatrix} Y - E E^T& A Y + B S \\
    \ast & Y \end{bmatrix} \in \Sigma^{2n,\alt}[\Pi]_{\leq 2d} \nonumber \\
    & \begin{bmatrix} Z & C Y + D S \\ \ast & Y \end{bmatrix} \in \psd^{n+r}_+
    % \in \Sigma^{n+r,\alt}[\Pi]_{\leq 2d} 
    \nonumber \\
    & \gamma^2 - \Tr Z \geq 0
    % \in \Sigma^{1,\alt}[\Pi]_{\leq 2d}
    \nonumber \\
  & Y \in \psd^n_{++}, Z\in \psd^r_+, \ S \in \R^{m \times n}.  \nonumber
\end{align}

All \ac{SOS} problems will be presented for the Alternatives case due to the Gram large matrix sizes involved in the Full algorithms. The Full problem would restrict the same matrix to the cone $\Sigma^{2n}[\bpc]_{\leq 2d}$. 
% \Return $H_2 = \gamma, \ K = SY^{-1}$ (or Infeasibility)
% \EndProcedure
% % \end{subequations}
% \end{algorithmic}
% \end{algorithm}
% \urg{Algorithm \ref{alg:h2} is not right, see comments}

% \begin{rem}
% The worst-case infinite-horizon Linear Quadratic Regulator (LQR) with weighting matrices $Q \in \psd_{++}^n, \ R \in \psd_{++}^m$ problem is an $H_2$ formulation with $C = Q^{\frac{1}{2}}$ and $D = R^{\frac{1}{2}}$.
% \end{rem}

%% file: sections/complexity.tex
\section{Computational Complexity}

% \urg{Complexity for quad/true as well}

\label{sec:complexity}

This section tabulates the computational complexity involved in executing (Full or Alternative) \ac{WSOS} tightenings of the stabilization programs \ref{prob:super_full}, \ref{prob:quad_full} \ref{prob:pos_full} under measurement noise. As a reminder, the Full methods involve $p_F = n(n+m+T)$ variables $(A, B, \dx)$, while the Alternatives methods employ $p_A = n(n+m)$ variables $(A, B)$ after eliminating $\dx$.

The following tables will use the notation $\R(\cdot)$ for the size (cone) of a vector $\R^{(\cdot) \times 1}$ (free variable) and $\psd_+(\cdot)$ for the size of a matrix $\psd_+^{(\cdot)}$ (semidefinite variable).

% \subsection{Superstabilization}
% \subsection{Stabilization}

The extended superstabilizing algorithms \eqref{eq:super_full_wsos} and \eqref{eq:super_altern_wsos}  involve $2n^2 + n$ polynomials $q(A, B)$ that must be nonnegative. The positive stabilizing algorithm \eqref{eq:pos_sos} and its Alternatives restriction has $n^2+n$ polynomials $q(A, B)$ that must be nonnegative.Table \ref{tab:size_ss} compares the variable sizes in the Putinar Psatz \eqref{eq:putinar} and the Alternatives Psatz \eqref{eq:altern_sos}. 
% \todo{Tables are way too small. Imposible to read}

\begin{table}[h]
    \centering
        \caption{Size of Superstabilizing Method (Psatz)}
    \label{tab:size_ss}
   % \begin{adjustbox}{width=0.48\textwidth}
   \begin{tabular}{|c|c|c|c|}
\hline
           & Number of Polynomials & Full                              & Alternatives                                             \\ \hline
$q$        & 1          & $\R{\binom{p_F + 2d}{2d}}$        & $\R{\binom{p_A + 2d}{2d}}$                          \\ \hline
$\sigma_0$ & 1          & $\psd_+{\binom{p_F + d}{d}}$      & $\psd_+{\binom{p_A + d}{d}}$                        \\ \hline
$\sigma_i$ & $2nT$      & $\psd_+\binom{p_F + d - 1}{d- 1}$ & $\psd_+\binom{p_A + d}{d}$                          \\ \hline
$\mu_j$    & $2n(T-1)$  & $\R{\binom{p_F + 2d-2}{2d-2}}$    & \multicolumn{1}{c|}{$\R{\binom{p_A + 2d-1}{2d-1}}$} \\ \hline 
\end{tabular}
\end{table}

The quadratically stabilizing programs in algorithm \eqref{eq:quad_sos} involves a single \ac{PMI} of size $2n$, along with a positive definite matrix $Y \in \psd^n_{++}$.
Let $\nu = n(n+1)/2$ be the number of free variables in a symmetric matrix in $\psd^n$. Table \ref{tab:size_qs} lists the size of the matrices involved in the quadratically stabilizing Psatz expressions \eqref{eq:scherer} and \eqref{eq:altern_sos}.

\begin{table}[h]
    \centering
        \caption{Size of Quadratically Stabilizing Method (Scherer Psatz)}
    \label{tab:size_qs}
    % \begin{adjustbox}{width=0.48\textwidth}
    \begin{tabular}{|c|c|c|c|}
\hline
       & Number of Polynomials & Full                                & Alternatives                      \\ \hline
$q$        & 1                     & $\R{\nu\binom{p_F + 2d}{2d}}$       & $\R{\nu\binom{p_A + 2d}{2d}}$     \\ \hline
$\sigma_0$ & 1                     & $\psd_+2n{\binom{p_F + d}{d}}$      & $\psd_+2n{\binom{p_A + d}{d}}$    \\ \hline
$\sigma_i$ & $2nT$                 & $\psd_+2n\binom{p_F + d - 1}{d- 1}$ & $\psd_+2n\binom{p_A + d}{d}$      \\ \hline
$\mu_j$    & $2n(T-1)$             & $\R{\nu\binom{p_F + 2d-2}{2d-2}}$   & $\R{\nu\binom{p_A + 2d-1}{2d-1}}$ \\ \hline
\end{tabular}
    % \end{adjustbox}
\end{table}

% The 

% It is of great importance to provide a computational complexity analysis which clarify the advantage of using alternatives method as opposed to Full method. We first analyze the complexity of Full method. From \eqref{eq:full_nonneg}, we have $n^2+n$ scalar polynomials $q(A,B,\delta x)\in \Sigma[x]_{2d}$ where each polynomial has $\binom{p + 2d}{2d}$ scalar variables with $p = n(n+m+T)$. Now from \eqref{eq:putinar} and \eqref{eq:full_con_sos}, on the right hand side, we have one $\sigma_0 \in \Sigma[x]^{s_0}$ where $s_0 = \binom{p + d}{d}$. $N_g = 2nT$, $\sigma_i \in \Sigma[x]^{s_i}$ where $s_i = \binom{p + \lfloor \frac{2d-d_g}{2}\rfloor}{\lfloor \frac{2d-d_g}{2}\rfloor}$. Correspondingly, $\mu_j$ is a polynomial vector of size $2n(T-1)$ with $\binom{p + 2d-d_h}{2d-d_h}$ scalar variables. 
\begin{rem}
The dual variables $\mu$ have a degree of $2d-2$ in the Full formulation $(A \dx)$ and a degree of $2d-1$ in the alternatives formulation.
\end{rem}

The main reductions in computational complexity are:
\begin{enumerate}
    \item The Alternatives Gram matrix sizes are independent with respect $T$.
    \item Experimental results demonstrate that the Full method often only works with $d \geq 2$, while the Alternatives methods can permit convergence with $d = 1$.
\end{enumerate}

% \old{

% Two major sources of computational complexity reduction are:\\
% (a)\; in Alternatives method, the number of variables $p$ does not depend on the number of samples $T$. \\
% (b)\; for efficiency, we would like to take the lowest order relaxation, i.e. the smallest $d$ such that the algorithm is feasible. Experimental results shows that Full method only works with $d \geq 2$ while the Alternatives method works with $d \geq 1$. \\
% % \begin{rem}
% % The multipliers $\mu$ against consistency constraints \eqref{eq:noise_bilinear} have degree $2d-2$ in Full (bilinearity $A \dx$) but have degree $2d-1$ in Alternatives (affine in $(A, B)$ after eliminating $\dx$).
% % \end{rem} 
% }

Table \ref{tab: size} displays the size (but not multiplicities) of variables involved in the Superstabilizing Algorithms \eqref{eq:super_full_wsos}, \eqref{eq:super_altern_wsos}.
Table \ref{tab: size_qs} lists sizes of \ac{SOS}-tightenings of the Quadratically  Algorithms \ref{eq:quad_sos} for parameters of $n=2, m=1, d_{full} = 2, d_{altern} = 1$ and increasing $T$. 
\begin{table}[h]
    \centering
        \caption{Size of Variables for Extended Superstability or Positive Stability}
    \label{tab: size}
    % \begin{adjustbox}{width=0.3\textwidth}
    \begin{tabular}{ |c|c|c|c|c| }
 \hline
  & $q$ & $\sigma_0$ & $\sigma_i$ & $\mu_j$\\ \hline
  Alternatives & 28 & 7 & 7 & 7 \\ \hline
 Full(T = 4)  & 3060 & 120 & 15 & 120 \\ \hline
 Full(T = 6) & 7315 & 190 & 19 & 190 \\ \hline
 Full(T = 8) & 14950 & 276 & 23 & 276 \\ \hline
    \end{tabular}
    % \end{adjustbox}
\end{table}

\begin{table}[H]
    \centering
        \caption{Size of Variables for Quadratic Stability}
    \label{tab: size_qs}
    % \begin{adjustbox}{width=0.3\textwidth}
    \begin{tabular}{ |c|c|c|c|c| }
 \hline
  & $q$ & $\sigma_0$ & $\sigma_i$ & $\mu_j$\\ \hline
  Alternatives & 280 & 28 & 28 & 70 \\ \hline
 Full(T = 4)  & 30600 & 480 & 60 & 1200 \\ \hline
 Full(T = 6) & 73150 & 760 & 76 & 1900 \\ \hline
 Full(T = 8) & 149500 & 1104 & 92 & 2760 \\ \hline
    \end{tabular}
    % \end{adjustbox}
\end{table}
\begin{rem}
Quadratic stability is much more costly to enforce compared with extended superstability or positive stability. The full method quickly becomes intractable as the number of samples increases. 
\end{rem}

% \subsection{Optimal Control}

The $H_2$ optimal control program in Algorithm \ref{eq:h2_wsos} has the same PSD matrix sizes as in Table \ref{tab: size_qs} for the Scherer constraint. It is additionally required that $Y \in \psd^n_{++}$ and $Z \in \psd^r_+$, but these sizes of $n$ and $r$ are small in comparision to the sizes in Table \ref{tab: size_qs}.

%% file: sections/all_noise.tex
\section{All Noise Sources}
\label{sec:all_noise}

This section reinserts the process noise $w$ and input noise $\du$ from the model \eqref{eq:model} into the description of plant consistency sets.

\subsection{Set Description and Full}

Let $(\epsilon_x, \epsilon_u, \epsilon_w) \geq  0$ be $L_\infty$ bounds for the measurement, input, and process noise respectively. The
data $\dc = (\tilde{x}_t, \tilde{u}_t)$ produces a consistency set $\bpc^{\all}$ as
\begin{align}
    \bar{\pc}^{\all}: \ \begin{Bmatrix*}[l]\norm{\dx_t}_\infty &\leq \epsilon_x,  & \ \forall t = 1..T \\
     \norm{\du_t}_\infty &\leq \epsilon_u, & \ \forall t = 1..T-1 \\
     \norm{w_t}_\infty &\leq \epsilon_w, & \ \forall t = 1..T-1 \\
     \end{Bmatrix*},
     \label{eq:pcbar_all}
\end{align}
with the defined quantities for all $t = 1..T-1$:
\begin{subequations}
\begin{align}
    0 &= -\dx_{t+1} + A \dx_t + B \du_t + w_t  + h^\all_t \\
    h^\all_t &= \tilde{x}_{t+1} - A \tilde{x}_{t} - B \tilde{u}_t.
\end{align}
\end{subequations}
The set $\bpc^\all$ in \eqref{eq:pcbar_all} is described by
% $2(T(2n+m)-n)$ 
$2((T-1)(2n+m)+n)$ polynomial inequality constraints and $n(T-1)$ linear equality constraints in terms of the $p^\all_F = n(n+m) + (T-1)(2n+m)+n$ variables $(A, B, \dx, \du, w)$.
Just as in \eqref{eq:pcbar} with $\pc$ and $\bpc$, the semialgebraic set of consistent plants $\pc^\all(A, B)$ may be formed by the projection,
\begin{align}
    \pc^\all(A, B) = \pi^{A,B} \bar{\pc}^\all(A, B, \dx, \du,w). \label{eq:pc_all}
\end{align}

The combination of measurement, input, and process noise may be incorporated into the Full algorithms \eqref{eq:super_full_wsos} (superstability),  \eqref{eq:quad_sos} (quadratic stability), and \eqref{eq:pos_sos} (positive stability) by imposing Psatz  (\eqref{eq:putinar} or \eqref{eq:scherer}) positivity constraints over the set $\bpc^\all$ in \eqref{eq:pcbar_all}.

\begin{rem}
We note that structured process noise $x_{k+1} = A x_k + B u_k + E w_k$ for $E \in \R^{n \times e}$ may be incorporated into the All-noise framework. 
Additionally, the process noise variables $w$ may be eliminated when $e \leq n$. Define a left inverse $E^\dagger$ with $E^\dagger E = I_e$  and a matrix $\mathcal{N}$ containing a basis for the nullspace of $E^T$ in its columns. The following equations may be imposed to represent the process noise constraints for all time indices $t \in 1..T-1:$
\begin{align*}
    \epsilon_w &\geq \norm{E^{\dagger} \left((\tilde{x}_{t+1} - \dx_{t+1}) - A(\tilde{x}_t - \dx_t) - B \tilde{u}_t \right)}_\infty \\
    0 &=  \ \mathcal{N} \left((\tilde{x}_{t+1} - \dx_{t+1}) - A(\tilde{x}_t - \dx_t) - B \tilde{u}_t \right).
\end{align*}
\end{rem}

\subsection{Alternatives}

The \ac{BSA} set \eqref{eq:pcbar_all} is described by $n(n+m)+T(2n+m)-n$ variables $(A, B, \dx, \du, w)$. The variables $(\dx, \du, w)$ may be eliminated by following the Theorem of Alternatives laid out in Section \ref{sec:altern}. 
The process noise variables $w_t$ under the constraint $\forall t=1..T-1, \ \norm{w_t}_\infty \leq \epsilon_w$ will be eliminated by rearranging terms in \eqref{eq:noise_bilinear}:

\begin{align}
      \label{eq:noise_bilinear_w}  w_t  &= A (\hat{x}_t-\dx_t) + B(u_t-\Delta u_t) - (\hat{x}_{t+1} - \dx_{t+1}) \nonumber\\
      &= A \dx_t + B \du_t + h^\all_t - \dx_{t+1}.
      \end{align}

A certificate of \ac{PSD}-ness for the following matrix function $q(A, B)$ will be derived (from \eqref{eq:altern_nonneg}):
\begin{align}
    q(A, B) &\in \psd_{++}^s & \forall (A, B) \in \pc^\all.
    \label{eq:altern_nonneg_all}
\end{align}

Dual variables $\mu^\pm \in (\psd^s_+[A,B])^{n \times (T-1)}$ for $\epsilon_w$, $\psi^\pm \in (\psd^s_+[A,B])^{n \times T}$ for $\epsilon_u$ and $\zeta^\pm \in (\psd^s_+[A,B])^{n \times T}$ for $\epsilon_x$ may be initialized according to the Putinar multipliers \eqref{eq:putinar_variables} to form the weighted sum $\Phi^\all$ with 
\begin{align}
    \Phi^\all &= -q(A, B) + \textstyle \sum_{t=1,i=1}^{T, n} \left(\zeta_{ti}^+(\epsilon_x - \dx_t) +\zeta^-_{ti}(\epsilon_x + \dx_t)\right) \nonumber \\
    &+ \textstyle \sum_{t,i}^{T, m} \left(\psi_{ti}^+(\epsilon_u - \du_t) +\psi^-_{ti}(\epsilon_u + \du_t)\right) \label{eq:Sij_all} \\
    &+\textstyle\sum_{t,i}^{T-1,n} \mu_{ti}^-(\epsilon_w - (A_i \dx_{ti}+ B_i \du_{ti} + h_{ti}^\all-\dx_{t+1, i} )) \nonumber \\
    &+\textstyle\sum_{t,i}^{T-1,n} \mu_{ti}^+(\epsilon_w + (A_i \dx_{ti} + B_i \du_{ti} + h_{ti}^\all-\dx_{t+1, i} )). \nonumber
\intertext{The $(\dx,\du)$-constant terms of $\Phi^\all$ are}
    Q^\all &= -q(A, B) + \textstyle\sum_{t=1,i=1}^{T-1,n} \epsilon_w(\mu^- + \mu^+) \nonumber\\
    &+\textstyle\sum_{t=1,i=1}^{T-1,n} h_{ti}^\all (\mu^+ - \mu^-) \label{eq:altern_const_all}\\
    &+\textstyle \sum_{t=1,i=1}^{T, m} \epsilon_u \left(\psi_{ti}^+  + \psi_{ti}^- \right) +\textstyle \sum_{t=1,i=1}^{T, n}\epsilon_x \left(\zeta_{ti}^+ + \zeta_{ti}^- \right). \nonumber
\end{align}

Define the following symbol $\Delta \zeta = \zeta^+ - \zeta^-$, with a similar structure holding for $\Delta \psi$ and $\Delta \mu$. Following the supremum $\leq 0$ procedure of Section \ref{sec:altern} leads to an alternatives-based nonnegativity certificate of \eqref{eq:altern_nonneg_all}:
\begin{subequations}
\label{eq:altern_feas_all}
\begin{align}
    \find_{\zeta, \mu, \psi} \ & -Q^\all(A, B; \  \zeta^\pm, \mu^\pm, \psi^\pm) \in \psd_+^s  & & \forall (A, B)\label{eq:altern_feas_all_Q}\\
&\Delta \zeta_{1i} = \textstyle \sum_{j=1}^n A_{ji}  \left(\mu_{1j}^+ - \mu_{1j}^-\right)  \\
& \Delta \psi_{ti} = B \Delta \mu_{ti} & & \forall t = 1..T\\
& \Delta \zeta_{ti} = \textstyle \sum_{j=1}^n A_{ji}  \Delta \mu_{tj}   -  \Delta \mu_{t-1,i} & & \forall t \in 2.. T- 1 \\
      & \Delta \zeta_{Ti}  = -\Delta  \mu_{T-1, i} \\
        & \psi_{ti}^\pm, \ \zeta_{ti}^\pm \in \psd_+^s[A,B] \label{eq:altern_feas_all_psi} & & \forall i,  \ \forall t \in 1..T \\
    & \mu_{ti}^\pm \in \psd_+^s[A,B]  & &\forall i,  \ \forall t \in 1..T-1. \label{eq:altern_feas_all_exists}
    \end{align}
\end{subequations}

The certificate \eqref{eq:altern_feas_all} involves only the $n(n+m)$ variables $(A, B)$ at the cost of requiring $2T(2n+m)-2n + 1$ Scherer Psatz constraints in $(A, B)$ (\eqref{eq:altern_feas_all_Q}  and  \eqref{eq:altern_feas_all_psi}-\eqref{eq:altern_feas_all_exists}). The cone of matrix-valued polynomials that admit certificates in \eqref{eq:altern_feas_all} may be expressed as $\Sigma^s[A,B]^{\alt,\all}$. The cone $\Sigma^n[A,B]^{\alt,\all}_{\leq d}$ may be substituted in for $\Sigma^s[A,B]^{\alt}_{\leq d}$ in all \ac{WSOS} tightenings such as 
Algorithm \eqref{eq:super_altern_wsos}. (superstability). An Archimedean set $\Pi^\all \supseteq \pc^\all$ must be known to ensure convergence of the Alternatives certificate as the degree $d \rightarrow \infty$ (from Assumption \ref{assum:compact}). All Scherer constraints would then take place under $-Q, \ \mu_{ti}^\pm, \ \zeta_{ti}^\pm \ \psi_{ti}^\pm \in \Sigma^s[\Pi^\all]_{\leq 2d}$ at degree-$d$.

%% file: sections/examples.tex
\section{Numerical Examples}
\label{sec:examples}

MATLAB (2021a) code to generate the below examples is publicly available\footnote{
\url{https://github.com/jarmill/error_in_variables}}. Dependencies include Mosek \cite{mosek92} and YALMIP \cite{Lofberg2004}.

In this section, we first discuss the performance of the all three  stabilization algorithms. This performance is empirically compared using Monte Carlo experiments by adjusting the noise level and number of samples of data in $\dc$. 
% One example is presented where superstabilization algorithm fails while quadratic stabilization algorithm works. 
Next, we show the worst-case optimal control for $H_2$  performance. A followup experiment shows the result where all type of noises are considered. Finally, we illustrate that the partial information helps identify a controller. 

\subsection{Monte Carlo Simulations for Stabilization}
To test the reliability of the proposed method, we collected 50 trajectories with different level of noise and choose $u, x_0$ to be uniformly distributed in $[-1, 1]$. We then apply Alternatives \ac{WSOS} tightenings of \eqref{eq:super_full}, \eqref{eq:quad_full}, and \eqref{eq:pos_full} on the following open-loop unstable system:
\begin{equation} \label{eq: sys}
A = \begin{bmatrix}  0.6863  & 0.3968\\
     0.3456  & 1.0388 \end{bmatrix}, \quad
B = \begin{bmatrix}  0.4170 & 0.0001 \\
    0.7203 & 0.3023 \end{bmatrix}.
\end{equation}
We first investigate the effect of noise by fixing $T = 8$. The result is reported in TABLE \ref{tab: eg1} showing the percentage of successful design for extended superstability (ESS), superstabilility (SS), positive stability (PS), and quadratic stability (QS).
\begin{table}[h]
    \centering
    \caption{Success rate $(\%)$ as a function of $\epsilon$ with $T = 8$}
    \label{tab: eg1}
    % \begin{adjustbox}{width=0.3\textwidth}
    \begin{tabular}{ |c|c|c|c|c| }
 \hline
 $\epsilon$ & 0.05 & 0.08 & 0.11 & 0.14\\ \hline
 % SS & 100 & 84 & 66 & 34 \\ \hline
 % QS & 100 & 100 & 80 & 58 \\ \hline
 % PS & 44 & 16 & 2 & 0 \\ \hline
 ESS & 100 & 88 & 69 & 40 \\ \hline
  SS & 100 & 84 & 57 & 39 \\ \hline
  PS & 94 & 61 & 19 & 3 \\ \hline
 QS & 100 & 100 & 90 & 79 \\ \hline

    \end{tabular}
    % \end{adjustbox}
\end{table}

As expected, ESS performs better than SS and PS individually, since ESS is a less restrictive stability condition. We note that QS was more successful than ESS in finding stabilizing controllers, but QS requires a maximal-size Gram matrix that is twice as large as in ESS. Increasing the noise level expands the consistency set, which in turn renders the problem of finding a single stabilizing controller more difficult. Collecting more sample data at the same noise bound $\epsilon=0.14$ reduces the size of the consistency set, as illustrated for all stabilization methods in TABLE \ref{tab: eg2}.

\begin{table}[h]
    \centering
        \caption{Success rate as a function of $T$ with $\epsilon = 0.14$}
    \label{tab: eg2}
    % \begin{adjustbox}{width=0.23\textwidth}
    \begin{tabular}{ |c|c|c|c|c| }
 \hline
 $T$ & 8 & 10 & 12 & 14 \\ \hline
 ESS & 40 & 61 & 78 & 89 \\ \hline
  SS & 39 & 60 & 75 & 86 \\ \hline 
 PS & 3 & 20 & 42 & 56 \\ \hline
 QS & 79 & 86 & 95 & 99 \\ \hline
 % SS & 34 & 54 & 70 & 86 \\ \hline
 % QS & 58 & 72 & 90 & 98 \\ \hline
    \end{tabular}
    % \end{adjustbox}
\end{table}

% Positive Stabilization is unable to find a solution at $\epsilon=0.14$ for the time ranges in Table \ref{tab: eg2}. To replicate the phenomenon of Table \ref{tab: eg2}, we sample 50 trajectories at $\epsilon=0.08$ from the following open-loop-unstable system:
% \begin{equation} \label{eq: sys_pos}
% A = \begin{bmatrix}  0.3026  & 1.0134\\
%      0.6426  & 0.3420 \end{bmatrix}, \quad
% B = \begin{bmatrix} -0.2308  & 0.0678 \\
%     -0.4592 & 0.8719 \end{bmatrix}.
% \end{equation}

% Table \ref{tab: eg3} reports the proportion of successful positive stable control designs of \eqref{eq: sys_pos}.

% \begin{table}[h]
%     \centering
%         \caption{Success rate as a function of $T$ with $\epsilon = 0.08$ on system \eqref{eq: sys_pos}}
%     \label{tab: eg3}
%     % \begin{adjustbox}{width=0.23\textwidth}
%     \begin{tabular}{ |c|c|c|c|c| }
%  \hline
%  $T$ & 8 & 10 & 12 & 14 \\ \hline
%  PS & 66 & 92 & 98 & 100 \\ \hline
%     \end{tabular}
%     % \end{adjustbox}
% \end{table}

% \begin{exmp}
To see an advantage of QS and ESS  over SS and PS, we now consider a simple mass-spring-damper system shown below with $x_1 = x, x_2 = \dot{x}, u = F$ that includes integrator dynamics:
% \end{exmp}
\begin{center}
\begin{equation}\label{eq: sys2}
    A = \begin{bmatrix} 0 & 1 \\ -\frac{k}{m} & -\frac{b}{m} \end{bmatrix}, \;\;\;
    B = \begin{bmatrix} 0 \\ \frac{1}{m} \end{bmatrix}
\end{equation}
\begin{figure}[!h]
\centering
\begin{circuitikz}
%ground
\usetikzlibrary{patterns}
\pattern[pattern=north east lines] (0,0) rectangle (7,.25);
\draw[thick] (0,.25) -- (7,.25);

\draw (3,.25) to[spring, l=$k$] (3,2);
\draw (4,.25) to[damper, l_=$b$] (4,2);
\draw[fill=gray!40] (2.5,2) rectangle (4.5,3);
\node at (3.5,2.5) {$m$};

\draw[thick, ->] (3.5,4) -- (3.5,3);
\node at (3.75,3.75) {$F$};

\draw[very thick,
    |-latex] (2,2.5) -- ++(0,-1)
    node[midway,left]{\small $x$};
\end{circuitikz}
    \caption{Illustration of a Spring-Mass-Damper System with parameters in \eqref{eq: sys2}}
    \label{fig:spring}
\end{figure}
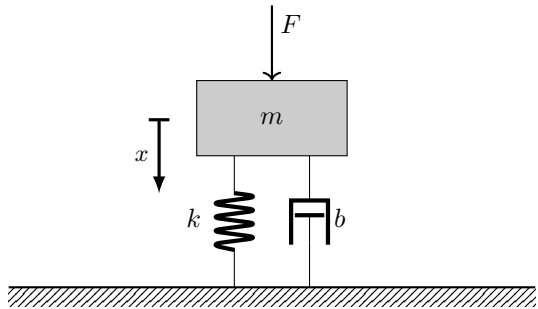
\end{center}
This system is not superstabilizable since any state feedback $u=Kx$ cannot affect the first row of $A$, hence the infinity norm of $A+BK$ is always greater or equal to $1$. Similarly, condition \eqref{eq:stab_clean_d_geq} for positive stabilization would require that $v_1 - v_1 = 0$ must be positive. 
System \eqref{eq: sys2} is extended-superstabilizable by \cite[Example 4]{polyak2004extended}, and 
one can easily apply an Alternatives \ac{WSOS} tightening to \eqref{eq:quad_full} to try and find a quadratically stabilizing controller at the cost of multiplying the size of all Gram matrices involved by 2.

\subsection{Worst-Case-Optimal Control}
To analyze the optimal performance of the proposed method, we first solve the standard $H_2$ problem \eqref{eq:h2_single} ($C = I_{r \times n}, \ D = [\0_{n, m}; I_m], \ E = I_{n \times e}$),  with known $A,B$ defined by \eqref{eq: sys}. We denote the benchmark as $\gamma_2 = 1.9084$. Now we apply Program \ref{eq:h2_wsos} to 50 trajectories with different level of noise. The effect of noise is shown in TABLE \ref{tab: eg4} with fixed $T = 8$.

For each of the 50 trajectories, Program \ref{eq:h2_wsos} returns a control policy $K$ and a worst-case-$H_2$ upper bound $\gamma_{2,worst}$ (valid for all $(A, B) \in \pc$). The quantity $\gamma_{2,clp}$ (closed-loop poles) is the $H_2$ norm found by applying $K$ as an input to the ground-truth system and solving Problem \ref{prob:h2_single}. It therefore holds that $\gamma_{2,worst} \geq \gamma_{2, clp}$ for each trajectory. The quantities returned in Table \ref{tab: eg4} and all subsequent tables are the median values of $(\gamma_{2, clp}, \gamma_{2,worst})$ over the 50 trajectories in order to prevent outliers deviating results.

% $\gamma_{2,clp}$ is the $H_2$ norm of the closed-loop system after we implement the data-driven controller from Algorithm \ref{alg:h2} to the true system, $\gamma_{2,worst}$ is the worst-case 
% $H_2$ norm over all possible systems
% in the consistency set. We used median instead of mean so as to prevent the outlier deviating the result. 
\begin{table}[h]
    \caption{$H_2$ performance as a function of $\epsilon$ with $T= 8$}
    \label{tab: eg4}
    \centering
    % \begin{adjustbox}{width=0.4\textwidth}
    \begin{tabular}{ |c|c|c|c|c| }
 \hline
 $\epsilon$ & 0.05 & 0.08 & 0.11 & 0.14 \\ \hline
 $\gamma_{2,clp}$ & 1.9684  &  2.0715 &   2.1773  &  2.1456 \\ \hline
 $\gamma_{2,worst}$ & 2.3004  &  2.7308 &   3.2279  &  4.3137 \\ \hline
    \end{tabular}
    % \end{adjustbox}
\end{table}

As we increase the noise, $\gamma_{2,worst}$ also increases since the consistency set is expanded. However, $\gamma_{2,clp}$ does not necessarily increase since we only optimize in the worst case. It is also worth noting that $\gamma_2 \leq \gamma_{2,clp} \leq \gamma_{2,worst}$ and the equality holds only if there is no noise. $H_2$ performance can be improved by collecting more samples as shown in TABLE \ref{tab: eg5}.

\begin{table}[h]
    \centering
        \caption{$H_2$ performance as a function of $T$ with $\epsilon = 0.08$}
    \label{tab: eg5}
    % \begin{adjustbox}{width=0.4\textwidth}
    \begin{tabular}{ |c|c|c|c|c| }
 \hline
 $T$ & 8 & 10 & 12 & 14 \\ \hline
 $\gamma_{2,clp}$ & 2.0715 & 1.9637 & 1.9373 & 1.9321 \\ \hline
 $\gamma_{2,worst}$ & 2.7308 & 2.4160 & 2.2328 & 2.2014 \\ \hline
    \end{tabular}
    % \end{adjustbox}
\end{table}

% \urg{$H_2$ takes 40-80s for each example while $H_\infty$ takes around 1 hour for each example. Not sure if we should put it.}
\subsection{All Noise}
The proposed framework is easily extended to handle different type of noises which better captures the estimation error of the collected data. Consider the following set of noise bounds:
\begin{equation}
\label{eq:all_noise_set}
    \begin{aligned}
    & \text{set} 1: \; \epsilon_x = 0.03, \epsilon_u = 0.00, \epsilon_w = 0.00 \\ 
    & \text{set} 2: \; \epsilon_x = 0.00, \epsilon_u = 0.02, \epsilon_w = 0.00 \\
    & \text{set} 3: \; \epsilon_x = 0.00, \epsilon_u = 0.00, \epsilon_w = 0.05 \\
    & \text{set} 4: \; \epsilon_x = 0.03, \epsilon_u = 0.02, \epsilon_w = 0.00 \\
    & \text{set} 5: \; \epsilon_x = 0.00, \epsilon_u = 0.02, \epsilon_w = 0.05 \\
    & \text{set} 6: \; \epsilon_x = 0.03, \epsilon_u = 0.02, \epsilon_w = 0.05 \\
    \end{aligned}
\end{equation}
For system \eqref{eq: sys}, we collected a single data trajectory of length $T = 8$ starting from an initial state of $x_1 = [1; 0]$ with $u$ uniformly distributed in $[-1,1]^2$. The worst-case $H_2$ norm for the all-noise bounds in \eqref{eq:all_noise_set} is collected in Table \ref{tab: eg6}:
\begin{table}[h]
    \centering
        \caption{$H_2$ performance for different sets of noise}
    \label{tab: eg6}
    % \begin{adjustbox}{width=0.47\textwidth}
    \begin{tabular}{ |c|c|c|c|c|c|c| }
 \hline
 set & 1 & 2 & 3 & 4 & 5 & 6\\ \hline
 $\gamma_{2,clp}$ & 1.9340 & 1.9131 & 1.9750 & 1.9615 & 2.1249 & 2.1659 \\ \hline
 $\gamma_{2,worst}$ & 2.0681 & 1.9628 & 2.1294 & 2.1554 & 2.5029 & 2.5973 \\ \hline
    \end{tabular}
    % \end{adjustbox}
\end{table}

% We added different set of noise defined above and apply the idea in Section \ref{sec:all_noise} to compute the convergence rate $\gamma_{SS}$ for SS method (Fig. \ref{fig: SS}) and the $H_2$ norm $\gamma_2$ for QS method (Fig. \ref{fig: H2}). 
% \begin{rem}
% Convergence rate is the $L_\infty$ norm of the system ranging from zero to one. It is a measure of the performance for SS method. The smaller $\gamma_{SS}$ is, the faster the system is. 
% \end{rem}
% \begin{figure}[h]
%     \centering
%     \includegraphics[width=8cm, height=5cm]{SS.pdf}
%     \caption{Convergence rate for different sets of noise}
%     \label{fig: SS}
% \end{figure}

% \begin{figure}[h]
%     \centering
%     \includegraphics[width=8cm, height=5cm]{H2.pdf}
%     \caption{$H_2$ norm for different sets of noise}
%     \label{fig: H2}
% \end{figure}

% The convergence rate for nominal plant is $0$, i.e. there exists a control gain to drive the system pole to the origin. 
The $H_2$ norm for the nominal plant is $\gamma_2 = 1.9084$. It is clear that adding more type of noise expands the consistency set hence leads to a larger worst-case $H_2$ norm. 

\subsection{Partial Information}
It is easy to incorporate the partial information in the proposed framework. Instead of treating all entries of $A,B$ as unknown variables, we can suppose $q$ entries of $(A,B)$ are known. There are now $n(n+m)-q$ free variables defining the consistency set, producing a smaller Gram matrix of $\binom{n(n+m)-q+d}{d}$ as compared to $\binom{n(n+m)+d}{d}$. This size reduction ensures that it is theoretically and computationally easier to find stabilizing controllers $K$ when partial information is known. For instance, if we assume that the first row of $A$ is known and apply Program \ref{eq:h2_wsos} with $T = 8$, $\epsilon = 0.08$, we get $\gamma_{2,clp} = 1.9568, \gamma_{2,worst} = 2.2566$ as compared to $\gamma_{2,clp} = 2.0715, \gamma_{2,worst} = 2.7308$ in the first column of TABLE \ref{tab: eg4}. 

We note that partial information can be incorporated into positive systems control. If it is a-priori known that $(A, B)$ arise from an internally positive system, then the elementwise nonnegativity constraints $A \in \R_{\geq0}^{n \times n}$ and $B \in \R_{\geq 0}^{n \times m}$ can be added to the consistency set $\bpc$ when applying \eqref{eq:pos_sos} or its Alternatives form.

%% file: sections/extensions_arxiv.tex
\section{Extensions}
This section outlines extensions to the proposed \ac{SOS}-based stabilization framework under measurement noise.

\label{sec:extensions}

\subsection{Time-Varying Noise Sets}

The measurement noise $\dx$ in $\bpc$ satisfies an $L_\infty$ constraint of $\norm{\dx_{t}} \leq \epsilon$. In principle, the Full method will converge to a superstabilizing or quadratically stabilizing controller when each $\dx_t$ is restricted to be a member of an Archimedean \ac{BSA} set. Such sets may vary in time or even in the plant parameters $(A, B)$ so long as $\bpc$ remains Archimedean. The Alternatives framework may be employed to simplify computational complexity when the sets $\mathcal{F}_t$ with $\dx_t \in \mathcal{F}_t, \ \forall t = 1..T$ are restricted to be polytopes.

% The constraint description for $\bpc$ 
% in equation \eqref{eq:pcbar} involves a noise bound of $\norm{\dx_{t}} \leq \epsilon$ for each $t = 1..T$. Time-dependent noise constraints may be developed by defining sets $\mathcal{F}_t$ such that 
% $\dx_{t} \in \mathcal{F}_t$. The \ac{WSOS} tightening of Program \eqref{eq:full_nonneg} will function when each $\mathcal{F}_t$ is \ac{BSA} in $\dx$. The Alternatives program in and its \ac{WSOS} tightening \eqref{eq:altern_nonneg} may be adapted when $\mathcal{F}_t$ are polytopes.

% These sets may additionally be plant dependent, and could be \ac{BSA} sets $F_t(A, B, \dx_t)$

\subsection{Non-uniform Sampling}

% The theorem of alternatives method for finding stabilizing controllers may be applied to scenarios where the state $x_t$ may not be available at every time instant $t \in 1..T$. 
% This extension will perform superstabilizing control in the missing-data  case with $(w, \du) = 0$ for ease of notation and derivation.
% \subsubsection{Sampling Derivation}
The polynomial-optimization-based framework for \ac{EIV} control may be performed in cases with missing state data.
This extension will pose consistency sets with $(w, \du) = 0$ for ease of notation and derivation.
Consider a trajectory with samples $(\hat{x}_t, \hat{x}_{t+2})$ and inputs $(u_t, u_{t+1})$, in which the state observation $\hat{x}_{t+1}$ is missing. Treating $x_{t+1}$ as an unknown variable would add a new source of uncertainty to the description of $\bpc$ and increase the complexity of posing Psatz constraints. The unknown state $x_{t+1}$ can be eliminated using the relation  $x_{t+1} = A x_{t} + B u_t$:
\begin{subequations}
\begin{align}
x_{t+2} &= A x_{t+1} + B u_{t+1} \\
x_{t+2} &= A (A x_{t} + B u_t) + B u_{t+1} \\
x_{t+2} &= A^2 x_t + A B u_t + B u_{t+1} \\
\hat{x}_{t+2} - \dx_{t+2} &= A^2 (\hat{x}_t -\dx_t) + A B u_t + B u_{t+1}.
\end{align}
\end{subequations}

% An affine expression in terms of $(\dx_{t+2}, \dx_t)$ may be formed,
% \begin{equation}
%     -\dx_{t+2} + A^2 \dx_t + (\hat{x}_{t+2} + A B u_t + B u_{t+1}).
% \end{equation}

Given the state observations $(\hat{x}_t, \hat{x}_{t+r})$ for $r \in \N, r >0$ and the input history ${u_k}_{k=t}^{t+r}$, the $\dx$-affine expression for the missing data case is
\begin{equation}
\label{eq:aff_missing_data}
    -\dx_{t+r} + A^r \dx_t + \left(\hat{x}_{t+r} - \textstyle \sum_{k = 0}^{r} A^{r-k} B u_{t+k}\right)=0.
\end{equation}

% A consistency set $\bpc$ may be formed from concatenating together constraints in \eqref{eq:aff_missing_data} along 
% These affine expressions assume that the input history is available between $u_t, \ldots, u_{t+r}$. 
Assume that state measurements are taken at $N_s$ times $\{t_s\}_{s=1}^{N_s} \subseteq (1..T)$. The missing-data records $\dc^{md}$ is comprised of the inputs $\{u_t\}_{t=1}^T$ and states $\{\hat{x}_{t_s}\}_{s=1}^{N_s}$. The missing-data consistency set is

\begin{align}
    \bar{\pc}^{md}: \ \begin{Bmatrix*}[l]\textrm{Eq \eqref{eq:aff_missing_data} with $t=t_s, \ r = t_{s+1}-t_s$}  & \forall k = 1..N_s-1 \\
     \norm{\dx_{t_k}}_\infty \leq \epsilon & \forall k = 1..  N_s\end{Bmatrix*}, \label{eq:pcbar_missing}
\end{align}

The Full Algorithms \eqref{eq:super_full_wsos}, \eqref{eq:quad_sos}, and \eqref{eq:pos_sos} may be used directly for missing-data by employing the consistency set $\bpc^{md}$ in  \eqref{eq:pcbar_missing} rather than $\bpc$.

An Alternatives Psatz from \ref{alg:altern_psatz} of a function $q(A, B) \in \psd_{++} \ \forall (A, B) \in \bpc^{md}$ may also be derived. 
The Superstability Alternatives program for missing-data involves two multipliers $\zeta^\pm_s$ for each sample $s \in 1, \ldots, N_s$.
Each sample $s$ from $1.. N_s-1$ inspires a multiplier $\mu_s$ and an $ h^0_k = \hat{x}_{t_{s+1}} - \textstyle \sum_{k = 0}^{r} A^{r-k} B u_{t_s+k}$ from \eqref{eq:aff_missing_data}. The Superstability Psatz from \eqref{eq:altern_feas} with non-uniform sampling is

\begin{subequations}
\label{eq:super_altern_nonuniform}
\begin{align}
&\exists \zeta_s^\pm \geq 0, \  \mu_s  \\
    & \textstyle q(A, B)\geq   \sum_{s, i} \epsilon(\zeta_{s, i}^+  +\zeta_{s, i}^-)  + \sum_{s=1}^{N_s-1} \mu_s^{N_s} h_s^0 \label{eq:altern_base} & & \forall (A, B)\\
&\zeta_1^+ - \zeta^-_1 = (A^{t_2 - t_1})^T \mu_1 \label{eq:altern_eq1} \\
      & \zeta_{s}^+ - \zeta^-_{s}= (A^{t_{s+1} - t_s})^T \mu_s - \mu_{s-1} & & \forall s\in 1..N_s- 1 \label{eq:altern_eqt} \\
      & \zeta_{N_s}^+- \zeta^-_{N_s} = - \mu_{N_s-1}. \label{eq:altern_eqT}
\end{align}
\end{subequations}

The Alternatives algorithm \eqref{eq:super_altern_wsos} may be used to perform superstabilizing control with the Psatz of \eqref{eq:super_altern_nonuniform}. A similar derivation can take place for Alternatives-based quadratic stabilization with  missing-data.

\subsection{Switched Systems Stabilization}

The final extension focuses on stabilization programs of switched systems.
Let there be a collection of $N_s$ discrete-time linear subsystems:
\begin{align}
\label{eq:hybrid}
    x_{t+1} =& A_s x_t + B_s u_t & \forall s = 1..N_s.
\end{align}

The switching sequence of the system \eqref{eq:hybrid} is a function $S: 1..T \rightarrow 1..N_s$ denoting the resident subsystem at time $t$.
An execution of the switched system \eqref{eq:hybrid} is a pair $\dc_S = (\{(x_t, u_t)\}_{t=1}^T, \{S_t\}_{t=1}^{T-1})$ such that,  
\begin{align}
\label{eq:hybrid_dyn}
    x_{t+1} =& A_{S_t} x_t + B_{S_t} u_t & \forall t = 1..T-1.
\end{align}

The execution $\dc_S$ is \textit{labeled} if the switching sequence $S_t$ is known. A labeled execution $\dc_S$ corrupted by $\epsilon$-bounded measurement noise $\dx$ inspires the following consistency set in terms of the unknown subsystem plants $\{A_s \in \R^{n\times n}, \ B_s \in \R^{n \times m}\}_{s=1}^T$ and measurement noise $\dx \in \R^{n \times T}$:
\begin{align}
    \bar{\pc}_S: \ \begin{Bmatrix*}[l]\forall t = 1..T-1: \\
    \quad \hat{x}_{t+1} - \dx_{t+1} = A_{S_t} (\hat{x}_{t}-\dx_t) + B_{S_t} u_t \\
     \forall t = 1..T: \qquad \norm{\dx_t}_\infty \leq \epsilon \end{Bmatrix*}. \label{eq:pcbar_switched}
\end{align}
A control $K$ that superstabilizes each subsystem $(A_s, B_s)$ may be found by solving Algorithm \eqref{eq:super_full_wsos}, in which all constraints are posed over the set \eqref{eq:pcbar_switched}, yielding a maximum PSD constraint of size $\binom{N_s n(n+m) + n T + d}{d}$ per Psatz.
Correlatively sparse cliques $(\dx_t, \dx_{t+1}, A_s, B_s)$ may be developed for each $t=1..T-1$ and $s=1..N_s$ \cite{waki2006sums}, with $N_s(T-1)$ instances of PSD constraints of maximal size $\binom{n(n+m+2) + d}{d}$ per Psatz. The Theorem of Alternatives may also be applied in order to eliminate the noise terms $\dx$, but the resultant polynomials $(\zeta^\pm, \mu)$ would have arguments consisting of all plant parameters $\{(A_s, B_s)\}_{s=1}^{N_s}$ with $2nT+1$ PSD constraints of maximal size $\binom{N_s n (n+m) + d}{d}$ per Psatz.

%% file: sections/conclusion.tex
\section{Conclusion}

This paper formulated a state-feedback stabilization problem for systems corrupted by $L_\infty$-bounded measurement, process, and input noise. 
\ac{WSOS} programs for superstabilizability,  quadratic stabilizability, and positive stability (Full) will converge to their respective controllers (if such a controller exists) as the degree tends towards infinity.
Such programs are computationally expensive with regards to the size of the \ac{PSD} matrices required in \acp{SDP}. A theorem of alternatives was deployed to create equivalent (superstability and positive stability) and conservative (quadratic stabilizability) programs at a reduced computational cost by eliminating the noise variables. This framework was extended to the robust control of plants in the consistency set, as laid out through $H_2$ methods.

Future work involves developing static-output-feedback and dynamic-output-feedback in the case of combined input noise and measurement noise. 
% This would  necessitates the design of observers for the process $\hat{y}_t = C x_t  + \Delta y_t$ where $C$ is non-invertible (and may even be unknown).
Other work involves analyzing conditions for which the sets $\bpc$ and $\pc$ are  compact in terms of the collected data $\dc$ (e.g. sampling complexity and  defining a concept of persistency of excitation), and quantifying the conservativeness involved with utilizing the Theorem of Alternatives in the $s > 1$ case. Tractability of this method would improve with further development and reductions in complexity of solving \acp{SDP}.

\label{sec:conclusion}
% \old{
% This work presented a convergent \ac{WSOS} program (Full) to perform superstabilization of systems with measurement noise. A theorem of alternatives was then applied to produce an equivalent problem (Alternatives) with reduced complexity when $M$ is constant in $\dx$. 
% % The Sparse Alternatives program is the most efficient method, but is no longer guaranteed to converge. 
% Efficacy of this method was demonstrated on example systems.

% % \urg{Conclude the paper. Summarize the problem and important findings. Add some extensions and future work.}

% % An extended Arxiv version of this paper is available at \urg{[Arxiv link goes here]}.
% Future work includes expanding the set of stable controllers beyond superstability, producing worst-case-LQR optimal controllers $K$ that can (super)stabilize consistent plants in $\pc$, and designing static output feedback controller-observers pairs that stabilize plants in $\pc$.

% % what will the arxiv version have that this work does not?}.
%     }

%% file: sections/acknowledgements.tex
\section*{Acknowledgements}

% J. Miller, T. Dai,  M. Sznaier were partially supported by the National Science Foundation (NSF) grants  CNS--1646121, CMMI--1638234, ECCS--1808381 and CNS--2038493,  and the Air Force Office for Scientific Research (AFOSR) grant FA9550-19-1-0005. 
% J. Miller was in part supported by the Chateaubriand Fellowship of the Office for Science \& Technology of the Embassy of France in the United States and the AFOSR International Student Exchange Program.

The authors would like to thank Didier Henrion, Victor Magron, Antonio Bellin, and the POP group at LAAS-CNRS for their advice and support.

%% file: sections/app_sos.tex
\section{Sum of Squares Methods}
\label{app:sos}
\subsection{Semialgebraic Geometry and Scalar Inequalities}
\label{sec:prelim_sos}
% \urg{Fill in SOS background for proofs of polynomial nonnegativity. Describe basic semialgebraic and semialgebraic sets.}

A \acf{BSA} set is the locus of a finite number of bounded-degree polynomial inequality and equality constraints \cite{lasserre2009moments}. A representation exists for every \ac{BSA} set $\mathbb{K}$ can be represented as
\begin{equation}
\mathbb{K} = \{x \mid g_i(x) \geq 0, \ h_j(x) = 0\},
\end{equation} for appropriate describing polynomials $\{g_i(x)\}_{i=1}^{N_g}$ and $\{h_j(x)\}_{j=1}^{N_h}$. 
The concatenation of describing polynomials implements the intersection of \ac{BSA} sets. 
The projection $\pi^x: (x, y) \mapsto x$ as applied to a \ac{BSA} set $\bar{\mathbb{G}}(x, y)$ is
\begin{equation}
    \mathbb{G}(x) = \pi^x \bar{\mathbb{G}}(x,y) = \{x \mid \exists y: (x, y) \in \bar{\mathbb{G}}\}.
\end{equation}

The projection of a \ac{BSA} set is not generally \ac{BSA}, and is instead termed as a semialgebraic set (closure of \ac{BSA} under projections, finite unions, and complements).

\ac{SOS} methods yield certificates that a polynomial $p(x)$ is nonnegative over a semialgebraic set $\K$. A polynomial $p(x)$ is \ac{SOS} ($p(x) \in \Sigma[x]$) if there exists a set of polynomials $\{q_i(x)\}_{i=1}^{N_q}$ such that $p(x) = \sum_{i=1}^{N_q} q_i(x)^2$.
To each \ac{SOS} certificate $\{q_i(x)\}_{i=1}^{N_q}$, there exists a nonunique choice of a size $s \in \N$, a \textit{Gram} matrix $Q \in \psd_{+}^s$, and a polynomial vector $v(x) \in \R[x]^s$, such that $p(x) = v(x)^T Q v(x)$. Given a matrix decomposition $Q = R^T R$, a valid certificate $\{q_i(x)\}_{i=1}^{N_q}$ may be recovered by $q(x) = R v(x)$. In practice, $v(x)$ is often chosen to be the degree-$d$ monomial map with $s = \binom{n+d}{d}$.

The set of \ac{SOS} polynomials with degree at most $2d$ is $\Sigma[x]_{\leq 2d} \subset \R[x]_{\leq 2d}$. Optimization over the cone of nonnegative polynomials is generically NP-hard, but tightening to the \ac{SOS} cone in fixed $(n, d)$ yields more tractable problems that can be solved using \acp{SDP} \cite{parrilo2000structured}.
By a theorem of Hilbert, the set of nonnegative polynomials in $n$ variables will coincide with \ac{SOS} polynomials with degree at most $2d$ only when $d=1$ (quadratic), $n=1$ (univariate), or $(n=2,d=2)$ (bivariate quartic) \cite{hilbert1888darstellung}. As the number of variables increases for fixed degree $d > 1$, the \ac{SOS} cone is less and less able to describe the set of nonnegative polynomials (in the sense of volume ratios) \cite{blekherman2006there}.

\ac{SOS} programming can be extended constrained optimization problems. The Putinar Positivestellensatz (Psatz) yields a sufficient certificate that a polynomial $p(x)$ is nonnegative over a \ac{BSA} $\mathbb{K}$
\cite{putinar1993compact}:
\begin{subequations}
\label{eq:putinar}
    \begin{align}
        & p(x) = \sigma_0(x) + \textstyle \sum_i {\sigma_i(x)g_i(x)} + \textstyle \sum_j {\phi_j(x) h_j(x)}\\
        &\exists  \sigma_0(x) \in \Sigma[x], \quad \sigma(x) \in (\Sigma[x])^{N_g}, \quad \phi \in (\R[x])^{N_h}. \label{eq:putinar_variables}
    \end{align}
\end{subequations}
The cone of polynomials that admit a Putinar certificate of nonnegativity over $\K$ in the sense of \eqref{eq:putinar} (where $\K$ is described by $g_i, \ h_j$) will be defined as the \ac{WSOS} cone $\Sigma[\mathbb{K}]$. The degree-$2d$ \ac{WSOS} cone $\Sigma[\K]_{\leq 2d}$ is the set of polynomials where multipliers in \eqref{eq:putinar_variables} are chosen such that each of the following products $(\sigma_0(x), \{\sigma_i(x) g_i(x)\}_{i=1}^{N_g}, \{\phi_j(x) h_j(x)\}_{j=1}^{N_h})$ has degree at most $2d$.

The \ac{BSA} set $\K \subset \R^n$ is \textit{compact} if there exists an $R \in (0, \infty)$ such that $\K \subseteq \{x \mid R - \norm{x}_2^2 \geq 0\}$ (Heine-Borel). It is additionally \textit{Archimedean} if  an $R \in (0, \infty)$ exists such that $R - \norm{x}_2^2 \in \Sigma[\K]$. Archimedeanness implies compactness, but there exists compact but non-Archimedean \ac{BSA} sets \cite{cimpric2011closures}. Given knowledge of an $R$ that proves compactness of $\K$, the compact set $\K$ can be rendered Archimedean by adding the redundant constraint $R - \norm{x}_2^2 \geq 0$ to $\K$'s description. Theorem 1.3 of \cite{putinar1993compact} states that every positive polynomial over an Archimedean $\K$ possesses a certificate \eqref{eq:putinar} (necessary and sufficient).
Increasing the degree $d$ of the \ac{WSOS} cones $\Sigma[\K]_{\leq 2d}$ until a valid positivity certificate is found is known as the moment-\ac{SOS} hierarchy \cite{lasserre2009moments}. The work in \cite{nie2007complexity} analyzes the convergence rate of the moment-\ac{SOS} hierarchy as $d$ increases for \acp{POP}.

Each iteration of an Interior Point Method solving an \ac{SDP} with a matrix variable $X \in \psd_+^N$ and $M$ affine constraint requires $O(N^3M + M^2 N^2)$ operations \cite{alizadeh1995interior}. Certifying \ac{SOS}-based nonnegativity of $p(x)$ over $\R^n$ through the degree-$d$ monomial map $v(x)$ requires a Gram $Q$ matrix of size $N = \binom{n+d}{n}$ and $M = \binom{n+2d}{2d}$ coefficient matching constraints. The per-iteration Interior Point computational complexity for fixed $n$ is therefore $O(d^{4n})$, and for fixed $d$ it is $O(n^{6d})$.

\subsection{Polynomial Matrix Inequalities}
\label{sec:pmi}
Letting $P(x) \in \psd^s[x]$ be a symmetric-matrix-valued polynomial, unconstrained and constrained \acp{PMI} posed over $P(x)$ are
\begin{align}
    P(x) &\succeq 0 & & \forall x \in \R^n, \label{eq:pmi_uncons}\\
    P(x) &\succeq 0 & & \forall x \in \K.\label{eq:pmi_cons}
\end{align}

This subsection (and paper) will assume that the \ac{BSA} set $\K$ is described by a set of scalar inequality and equality constraints, and that the only matrix constraints involved are $P(x) \succeq 0$.
Standard Putinar \ac{SOS} methods may be used to certify the above \acp{PMI} through the use of scalarization. Defining $y \in \R^s$ as a new variable, constraint \eqref{eq:pmi_uncons} is equivalent to imposing that the polynomial $y^T P(x) y$ is nonnegative for all $(x, y) \in \R^n \times \R^s$. Similarly, the constrained problem \eqref{eq:pmi_cons} can restrict $y^T P(x) y$ is nonnegative for all $(x, y)$ in the compact set $\K \times \{y \mid \norm{y}_2^2=1\}$. Applying scalarization increases the size of the Gram matrix used in the Putinar Psatz \eqref{eq:putinar} to $\binom{n+s+d}{d}$.

The work in \cite{scherer2006matrix} introduced a Psatz for \acp{PMI} with a typically reduced Gram matrix size of $s \binom{n+d}{d}$. A matrix $P(x) \in \psd^s[x]$ is an \ac{SOS}-matrix $P(x) \in \Sigma^s[x]$ if there exists a size $q \in \N$, a polynomial vector $v(x) \in \R[x]^q$, and a \textit{Gram} matrix $Q \in \psd_+^{q s}$ such that (Lemma 1 of \cite{scherer2006matrix}):
\begin{equation}
    P(x) = (v(x) \otimes I_s)^T Q (v(x) \otimes I_s).
\end{equation}
\ac{SOS} matrices and polynomials satisfy the relation $\Sigma[x] = \Sigma^1[x]$. The Scherer Psatz proving that a matrix $P(x)$ is \ac{PD} ($P(x) \succ 0$) over $x \in \K$ (for some $\varepsilon > 0$) is  \cite[Theorem 2]{scherer2006matrix} 
\begin{subequations}
\label{eq:scherer}
    \begin{align}
        & P(x) = \sigma_0(x) + \textstyle \sum_i {\sigma_i(x)g_i(x)} + \textstyle \sum_j {\phi_j(x) h_j(x)} + \varepsilon I_q\\
        &\exists  \sigma_0(x) \in \Sigma^s[x], \ \sigma_i(x) \in \Sigma^s[x], \ \phi_j \in \psd^s[x]. \label{eq:scherer_variables}
    \end{align}
\end{subequations}

% \todo{maybe also add number of variables} 

The Scherer representation \eqref{eq:scherer} is necessary and sufficient to certify $P(x) \succ 0$ over $\K$ if $\K$ is Archimedean (Remark 2 and Equation 9 of \cite{scherer2006matrix}). Note that the multipliers $\sigma, \mu$ from the Scherer Psatz \eqref{eq:scherer_variables} are symmetric-matrix-valued polynoimals of size $s$. The matrix-\ac{WSOS} cone $\Sigma^s[\K]$ and its degree-$2d$ truncation $\Sigma^s[\K]_{\leq 2d}$ will denote the cone of matrices in $\psd^n[\K]$ that admit Scherer positivity certificates in \eqref{eq:scherer}, just like $\Sigma[\K]$ and $\Sigma[\K]_{\leq 2d}$ in the scalar Putinar case. The SOS-matrix $\sigma_0(x)$ at the degree-$2d$ representation of \eqref{eq:scherer} has a Gram matrix of size $s \binom{n+d}{d}$, and the matrix $\sigma_0(x)$ may be described by $s(s+1)/2 \binom{n+2d}{2d}$ independent coefficients. Given degrees $\forall i: d_i = \ceil{\deg{g_i(x)}}, \ d \geq d_i$, the multipliers $\sigma_i(x)$ are polynomial matrices with Gram matrices of size $s \binom{n+d-d_i}{d-d_i}$.

%% file: sections/app_continuity.tex
\section{Continuity of Multipliers}
\label{app:continuity}

This appendix proves that the multiplier functions $(\zeta^\pm(A,B), \mu(A,B))$ from program \eqref{eq:altern_feas} (when feasible) may be chosen to be continuous for any \ac{PD} function $q(A,B)$ over $\pc$ (Theorem \ref{thm:dual_continuous}). 
This proof will establish that a set-valued map based on a variant of the feasible set of \eqref{eq:altern_feas} is lower semicontinuous, and will then apply Michael's theorem to certify continuous selections. The approach taken in this proof is similar to establishing lower semi-continuity of $\rho$-indexed sets from \cite{denel1979extensions}, but our problem has perturbations in the left-hand side multiplying the equality-multipliers $\mu$.

% \subsection{Set-Valued Maps}

Let $q(A, B): \pc \rightarrow \psd_{++}^s$ be a function that is \ac{PD} over the consistency set of plants $\pc$ with a certificate of nonnegativity $(\zeta^\pm(A,B), \mu(A,B))$ by program \eqref{eq:altern_feas}.
We note that $\pc$ is compact by Assumption \ref{assum:compact} and that the mapping from $\pc$ to the constraints in \eqref{eq:altern_feas} are affine (Lipschitz) in $(A, B)$ over the compact $\pc$.

Let $Z = (\psd^s)^{n \times T} \times (\psd^s)^{n \times T} \times (\psd^s)^{n \times (T-1)}$ be the residing space  (possible range) of the  multipliers $(\zeta^+(A,B), \zeta^-(A,B), \mu(A,B))$. In this appendix, the notation $(\zeta^\pm(A,B), \mu(A,B))$ will refer to functions from \eqref{eq:altern_feas}, and variables $(\zeta^\pm, \mu)$ lacking arguments $(A,B)$ will be values in $Z$.

A convex-set-valued map $S: \pc \rightrightarrows Z$ may be defined as the feasible set of program \eqref{eq:altern_feas} for each $(A, B) \in \pc$. The domain of $S$ is equal to $\pc$, since the functions $(\zeta^\pm(A, B), \mu(A,B))$ have values in $Z$ ($S(A,B) \neq \varnothing$) for all $(A, B) \in \pc$. The range is nonclosed due to the \ac{PD} constraint in \eqref{eq:altern_feas_Q}.
% \begin{lem}
% The domain of $S$ ($\textbf{dom} S$) is equal to $\pc$. Equivalently, $S(A, B) \neq \varnothing \ \forall (A, B) \in \pc$
% \end{lem}
% \begin{proof}
% % Given that the functions $(\zeta^\pm(A, B), \mu(A,B)): \pc \rightarrow Z$ are well-defined, their evaluations 
% \end{proof}

Define $\tau^* = \min_{(A,B) \in \pc} \lambda_{\min}(q(A,B))$ as the minimal possible eigenvalue of the \ac{PD} matrix $q(A,B)$. We note that the minimum is attained with $\tau^*>0$ because $\pc$ is compact and $\lambda_{\min}(q(A, B))$ is a continuous function of $(A, B)$ (given that the eigenvalues of a matrix are continuous in the matrix entries).
% \urg{prove that $\tau^* > 0$ because $\pc$ is compact and $\lambda_{\min}$ is continuous.}

We define the closed-convex-valued map $S^\tau: \pc \rightrightarrows Z$ for a value $0 < \tau < \tau^*$ as the set of solutions $(\zeta^\pm, \mu)\in Z$ for each $(A,B) \in \pc$ as
\begin{subequations}
\label{eq:altern_feas_tau}
\begin{align}
    & -Q(A, B; \zeta^\pm, \mu) - \tau I \in \psd^s_+ \label{eq:altern_feas_Q_tau}\\
&\zeta_{1i}^+ - \zeta^-_{1i} = \textstyle \sum_{j=1}^n A_{ji}  \mu_{1j}\label{eq:altern_feas_Q_diff_start}  \\
& \zeta_{ti}^+ - \zeta^-_{ti}= \textstyle \sum_{j=1}^n A_{ji}  \mu_{tj}   -  \mu_{t-1,i} & & \forall t \in 2.. T- 1 \\
      & \zeta_{Ti}^+ - \zeta^-_{Ti} = - \mu_{T-1, i} \label{eq:altern_feas_Q_diff_end}\\
      & \zeta_{ti}^\pm \in \psd_+^s, \ \mu_{ti} \in \psd^s & & \forall t=1..T. \label{eq:altern_feas_exists_tau_z}
\end{align}
\end{subequations}

The set-valued mappings $S$ and $S^\tau$ are related by $S^\tau \subset S$ for all admissible $\tau$.

Any (possibly discontinuous) solution function $(\zeta^\pm(A,B), \mu(A,B))$ that certifies positive-definiteness of $q(A, B)$ over $\pc$ via \eqref{eq:altern_feas} satisfies
\begin{align}
    \forall (A, B) \in \pc,  \tau \in (0, \tau^*): & & (\zeta^\pm(A,B), \mu(A,B)) & \in S^\tau(A, B).
\end{align}

A solution $(\zeta^\pm, \mu) \in S^\tau(A,B)$ is a \textit{Slater point} if all matrices in \eqref{eq:altern_feas_Q_tau} and \eqref{eq:altern_feas_exists_tau_z} are \ac{PD} and all equality constraints are fulfilled. A solution $(\zeta^\pm, \mu) \in S^{\tau'}(A,B)$ for some $\tau' > 0$ may be transformed into a new solution $(\tilde{\zeta}^\pm, \mu) \in S^{\tau/4}(A,B)$ such that $(\tilde{\zeta}^\pm, \mu)$ is a Slater point. Specifically, we express $-Q - \tau I$ from \eqref{eq:altern_const} as
\begin{subequations}
\begin{align}
    -Q - \tau I &= q(A, B) - \textstyle \sum_{t=1,i=1}^{T, n} \epsilon \left(\zeta_{ti}^+ + \zeta^-_{ti}\right)\nonumber\\ &-\textstyle\sum_{t=1,i=1}^{T-1,n}  \mu_{ti}h^{0}_{ti} - \tau I \\
    &= q(A, B) - \textstyle \sum_{t=1,i=1}^{T, n} \epsilon \left(\zeta_{ti}^+ + \zeta^-_{ti} + \frac{\tau}{2T n \epsilon}I\right) \nonumber\\ &-\textstyle\sum_{t=1,i=1}^{T-1,n}  \mu_{ti}h^{0}_{ti} - \frac{\tau}{2} I.
\intertext{The shifted multipliers $\tilde{\zeta^\pm}$ may be defined as}
    \tilde{\zeta}^\pm_{ti}&= \zeta^\pm_{ti} + \frac{\tau}{4T n \epsilon}I, \label{eq:multiplier_shift}\\
\intertext{which yields}
    -Q - \tau I&= q(A, B) - \textstyle \sum_{t=1,i=1}^{T, n} \epsilon \left(\tilde{\zeta}_{ti}^+ + \tilde{\zeta}^-_{ti}\right) \nonumber\\ &-\textstyle\sum_{t=1,i=1}^{T-1,n} \mu_{ti}h^{0}_{ti} - \frac{\tau}{2} I \\
    &= \textstyle -\tilde{Q} - \frac{\tau}{2} I. \label{eq:Q_shift}
    \end{align}
\end{subequations}

The differences in \eqref{eq:altern_feas_Q_diff_start}-\eqref{eq:altern_feas_Q_diff_end} cancel out with $\forall t, i:$,
\begin{align}
    \tilde{\zeta}^+_{ti} - \tilde{\zeta}^-_{ti} = \left(\zeta^+_{ti} + \frac{ \tau}{4 T n \epsilon}\right) - \left(\zeta^-_{ti} + \frac{ \tau}{4 T n \epsilon}\right) = \zeta^+_{ti} - \zeta^-_{ti}. \label{eq:difference_shift}
\end{align}
    
\begin{lem}
The solution $(\tilde{\zeta}^\pm(A,B), \mu(A,B))$ constructed from a certificate $(\zeta(A, B), \mu(A,B))$ from \eqref{eq:altern_feas} by \eqref{eq:multiplier_shift} is a Slater point of $S^{\frac{\tau}{4}}(A, B)$ for each $(A,B) \in \pc$.
\label{lem:shift_slater}
\end{lem}
\begin{proof}
Given that $\zeta^\pm_{ti}(A, B)$ is \ac{PSD} for each $(t,i)$ (from \eqref{eq:altern_feas_exists_zeta}),  adding  \iac{PD} matrix $\frac{\tau}{4 T n \epsilon} I$ to $\zeta^\pm_{ti}(A,B)$ will produce  \iac{PD} $\tilde{\zeta}^\pm_{ti}(A,B)$ from \eqref{eq:multiplier_shift}. The matrix $-\tilde{Q}$ from \eqref{eq:Q_shift} is also \ac{PD}, since $-\tilde{Q} - \frac{\tau}{2} I \succeq 0 \implies -\tilde{Q} \succ \frac{\tau}{4}I$. All equality constraints remain feasible by the observation in  \eqref{eq:difference_shift}, fulfilling the Slater point description.
\end{proof}

\begin{lem}
The set-valued mapping $S^{\tau}$ is lower-semicontinuous over $\pc$.
\label{lem:lsc}
\end{lem}
\begin{proof}
% This follows from Theorem 5.9a of \cite{rockafellar2009variational} (due to closed convex images of $S^{\frac{\tau}{4}}$ and $S^\tau$) and from 
This follows from the (strong) Slater point characterization points in $S^\tau$ within $S^{\frac{\tau}{4}}$ from Lemma \eqref{lem:shift_slater} by arguments from \cite{daniilidis2013lower} extended to the Matrix case, given that $S^\tau$ has closed convex images and sends a compact set to a Banach space.
\end{proof}

The condition for Michael's theorem (Thm. 9.1.2 of \cite{aubin2009set}) to hold, guaranteeing a continuous selection of $S^{\tau}$ is that $\pc$ is compact, $Z$ is a Banach space, $S^{\tau}$ is lower-semicontinuous, and $S^{\tau}$ has closed, nonempty, convex images for each $(A, B) \in \pc$. All of these conditions hold, so a continuous selection $(\zeta^\pm_s, \mu_s): \pc \rightarrow Z$ may be chosen with $(\zeta^\pm_s(A,B), \mu_s(A, B)) \in S^{\tau}(A,B) \subset S(A,B)$. These continuous functions $(\zeta^\pm_s, \mu_s)$ may therefore be used to certify positive-definiteness of $q(A, B)$ over $\pc$ in \eqref{eq:altern_feas}.

% The two set-valued maps hold the relation that $S^\tau \subset S$, 

% \urg{Finish this. Write up the note in \url{https://www.dropbox.com/s/ueisaz3lllr330w/Continuity\%20noise\%20in\%20variable.pdf?dl=0}}

% \old{
% The dual multipliers $\zeta_{ti}^\pm(A, B)$ and $\mu_{ti}(A, B)$ may be chosen to have continuous selections by following the arguments of Theorem \ref{lem:m_cont_select}
% % \urg{Wrong. figure out selections and when they apply}
% }

% \todo{Not sure if this is true. Solutions of LP are Lipschitz w.r.t changes in RHS, but here you are changing the matrix multiplying $\mu$ so Mangasarian's does not apply.} 

%% file: sections/app_polynomial_sw.tex
\section{Polynomial Approximability of Multipliers}
\label{app:polynomial_sw}

This appendix proves that \iac{PD} function $q(A,B)$ over $\pc$ may be certified using polynomial multipliers $(\tz^\pm(A,B), \tm(A,B))$ whenever \eqref{eq:altern_feas} is feasible (Theorem \ref{thm:dual_polynomial}). The proof will proceed through the introduction of three positive approximation tolerances $(\eta_0, \eta_1, \eta_2) > 0$ for use in the Stone-Weierstrass theorem over the compact set $\pc$.
For a matrix $F \in \R^{n \times m}$, define the element-wise maximum-absolute-value operator as $\mabs{F} = \max_{(i, j)} \abs{F_{ij}}$.

Let $(\zeta^\pm(A, B), \mu(A, B))$ be a continuous multiplier certificate from \eqref{eq:altern_feas}, in which continuity was proven by Appendix \ref{app:continuity}. 
A symmetric polynomial multiplier matrix $\tm_{ti} \in \psd^s[A, B]$ can be created for each $(t, i)$ using the Stone-Weiesrtrass theorem:
\begin{align}
\label{eq:poly_mu}
    \sup_{(A, B) \in \pc} \mabs{\tm_{ti}(A, B) - \mu_{ti}(A, B)} < \eta_0.
\end{align}

\subsection{Multiplier Bound}
This subsection will approximate the $\zeta$ multipliers by polynomials.
Let $\gamma_{ti} \in \psd[A, B]$ be the right-hand-sides of constraints \eqref{eq:altern_feas_zeta1}-\eqref{eq:altern_feas_zetaT} given $\tm$ with ($\forall i=1..n$):
\begin{subequations}
\begin{align}
    &\gamma_{1i} = \textstyle \sum_{j=1}^n A_{ji}  \tm_{1j}  \\
& \gamma_{ti}= \textstyle \sum_{j=1}^n A_{ji}  \tm_{tj}   -  \tm_{t-1,i} & & \forall t \in 2.. T- 1 \\
      & \gamma_{Ti} = - \tm_{T-1, i}. \end{align}
      \end{subequations}
      
The equations $\zeta_{ti}^+ - \zeta_{ti}^- = \gamma_{ti}$ from constraints \eqref{eq:altern_feas_zeta1}-\eqref{eq:altern_feas_zetaT} have solutions that can be parameterized by a set of continuous symmetric-matrix-valued functions $\phi_{ti}(A, B): \pc \rightarrow \psd^s$:
\begin{subequations}
\label{eq:gamma_phi_param}
\begin{align}
    \forall (t, i): \quad  \zeta_{ti}^+(A, B) &= \phi_{ti}(A, B)/2 + \gamma_{ti}(A, B)/2  \\
      \zeta_{ti}^-(A, B) &= \phi_{ti}(A, B)/2-\gamma_{ti}(A, B)/2.
\end{align}\end{subequations}

The functions $\phi_{ti}$ may be $\eta_1$-approximated by polynomials $\tilde{\phi}_{ti}(A, B) \in \psd^s[A, B]$ for each $(t, i)$ in the compact region $\pc$:
\begin{align}
\label{eq:poly_phi}
  \sup_{(A, B) \in \pc} \mabs{\tph_{ti}(A, B) - \phi_{ti}(A, B)} < \eta_1.
\end{align}

A tolerance $\eta_2>0$ is introduced to define the polynomial approximators $\tz_{ti}$ for each $(t, i)$:
\begin{subequations}
\label{eq:poly_zeta}
\begin{align}
    \tz_{ti}^+(A, B) &= \tph_{ti}(A, B)/2 + \gamma_{ti}(A, B)/2 + (\eta_2/2) I \\
      \tz_{ti}^-(A, B) &= \tph_{ti}(A, B)/2-\gamma_{ti}(A, B)/2 + (\eta_2/2) I.
\end{align}\end{subequations}

The approximators $\tz$ are related to the original multipliers $\zeta$ for each $(t, i)$ by
\begin{align}
\label{eq:poly_zeta_rel}
     \tz_{ti}^\pm(A, B) &= \zeta^\pm_{ti}(A, B) + (\tph_{ti}(A, B) - \phi_{ti}(A, B))/2  + (\eta_2/2) I.
\end{align}

The approximant $\tz$ must take on \ac{PSD} values in order to ensure that it is a valid multiplier for \eqref{eq:altern_feas_exists_zeta}. We utilize a result from the theory of interval matrices in order to choose $\eta_2$.

\begin{lem}
\label{lem:interval}
Let $M \in \psd^s_{++}$ and $R \in \psd^s$ be matrices with $\mabs{R} \leq \eta$ for some $\eta> 0$. A sufficient condition for $M + R \in \psd^s_{++}$ for all possible choices of $R$ is that $M - \eta s I \in \psd_{++}^s$.
\end{lem}
\begin{proof}
A symmetric interval matrix $G^I$ may be described by a center $C \in \psd^s$ and a radius $\Delta \in \psd^s$, in which the elements $G \in G^I$ satisfy $G_{ij} \in [C_{ij} - \Delta_{ij}, C_{ij} + \Delta_{ij}]$ and $G_{ij} = G_{ji}$. Letting $\rho(\Delta)$ be the spectral radius of $\Delta$ (maximum absolute value of eigenvalues), Theorem 5 of \cite{rohn1994positive} states that a sufficient condition for all symmetric interval matrices in the family $G^I$ to be \ac{PD} is that $\lambda_{min}(C) > \rho(\Delta)$.

Let $J_s$ be the all-ones matrix of size $s$.
The interval matrix $(M+R)$ for $\mabs{R} \leq \eta$ has a center of $C = M$ and a radius of $\Delta = \eta J_s$. The spectral radius of $\eta J_s$ is $\rho(\eta J_s) = \eta s$. The sufficient condition for Interval Positive-Definiteness as provided by Theorem 5 of \cite{rohn1994positive} is that $\lambda_{min}(M) > \eta s$, or equivalently $M - \eta s I \in \psd^s_{++}$.
\end{proof}

The matrix $(\eta_2/2) I - (\tph_{ti}(A, B) - \phi_{ti}(A, B))/2$ from \eqref{eq:poly_zeta_rel} may be treated as an interval matrix with center $C = (\eta_2/2) I$ and radius $\delta = (J_s \eta_1)/2$. Given that $\zeta_{ti}^\pm \in \psd_{+}^s(A, B)$, a sufficient condition for $\tz^\pm_{ti} \in \psd_{+}^s(A, B)$  by Lemma \ref{lem:interval} is
\begin{equation}
    \label{eq:eta2_eta1}\eta_2 > \eta_1 s.
\end{equation}

\subsection{Certificate Bound}

The $\dx$-constant term $Q$ in \eqref{eq:altern_const} has a polynomial approximation (when substituting $\zeta^\pm \rightarrow \tz^\pm, \ \mu \rightarrow \tm)$ of
\begin{subequations}
\label{eq:altern_const_poly}
\begin{align}
    \tilde{Q} &= -q(A, B) + \textstyle \sum_{t=1,i=1}^{T, n} \epsilon \left(\tz_{ti}^+ + \tz^-_{ti}\right) \nonumber\\
    &+\textstyle\sum_{t=1,i=1}^{T-1,n}  \tm_{ti}h^{0}_{ti} \\
    &= Q + \textstyle\sum_{t=1,i=1}^{T, n} \epsilon(\tph_{ti}- \phi_{ti}  + \eta_2 I) \nonumber\\  
    &+\textstyle\sum_{t=1,i=1}^{T-1,n} (\tm_{ti}-\mu_{ti})h^{0}_{ti} \\
    &=Q + \epsilon \eta_2 T n I +  \textstyle\sum_{t=1,i=1}^{T, n} \epsilon(\tph_{ti} - \phi_{ti}) \nonumber\\
    &+ \textstyle\sum_{t=1,i=1}^{T-1,n} (\tm_{ti}-\mu_{ti})h^{0}_{ti}. 
    \end{align}
    \end{subequations}
    
    The negative of \eqref{eq:altern_const_poly} is
    \begin{align}
    -\tilde{Q}&=-Q - \epsilon \eta_2 T n I -  \textstyle\sum_{t=1,i=1}^{T, n} \epsilon(\tph_{ti} - \phi_{ti})\nonumber \\
    &- \textstyle\sum_{t=1,i=1}^{T-1,n} (\tm_{ti}-\mu_{ti})h^{0}_{ti} \label{eq:altern_const_poly_neg}.
\end{align}

% Constraint \eqref{eq:altern_feas_Q} assumes that $-Q$ is \ac{PD} in $\pc$, and requires that $-\tilde{Q}$ is \ac{PD} in $\pc$. 

Interval matrices and Lemma \ref{lem:interval} will be used to derive a sufficient condition on $(\eta_0, \eta_1, \eta_2)$ such that $-\tilde{Q}$ is \ac{PD} for all $(A, B) \in \pc$ (satisfying condition \eqref{eq:altern_feas_Q}). Define the smallest eigenvalue of $-Q$ as
\begin{align}
    \lambda^* &= \min_{(A, B) \in \pc} \lambda_{min}(-Q(A, B)) > 0 \\
    \intertext{Because $Q$ satisfies \eqref{eq:altern_feas_Q}, $-Q$ is \ac{PD} over $\pc$ and therefore $\lambda^* > 0$. Further define $\bar{h}^0$ using \eqref{eq:aff_weight} as}
   \bar{h}^0_{ti} &= \max_{(A, B) \in \pc} h^0_{ti}(A, B), \quad  \forall (t, i).
\end{align}

The expression in \eqref{eq:altern_const_poly_neg} therefore satisfies
\begin{align}
    -\tilde{Q} &\succeq (\lambda^* - \epsilon \eta_2 T n)I -  \textstyle\sum_{t=1,i=1}^{T, n} \epsilon(\tph_{ti} - \phi_{ti}) \nonumber \\
    &- \textstyle\sum_{t=1,i=1}^{T-1,n} (\tm_{ti}-\mu_{ti})\bar{h}^{0}_{ti}. \label{eq:altern_const_bar}
\end{align}
Each $(\tph_{ti} - \phi_{ti})$ term in \eqref{eq:altern_const_bar} may be treated as an interval matrix with radius $\0$ and center $\eta_1 J_s$. Likewise, each $(\tm_{ti}-\mu_{ti})$ may be generalized to an interval matrix with radius $\0$ and center $\eta_0 J_s$. This interval matrix description of the right-hand-side of  \eqref{eq:altern_const_bar} leads to an interval matrix with center $C_Q$ and radius $\Delta_Q$:
\begin{subequations}
\label{eq:interval_Q}
\begin{align}
    C_Q &= (\lambda^* - \epsilon \eta_2 T n)I \\
    \Delta_Q &= \textstyle\sum_{t=1,i=1}^{T, n} \epsilon(\eta_1 J_s) + \sum_{t=1,i=1}^{T-1,n} (\eta_0 J_s)\bar{h}^{0}_{ti}. \\
    \intertext{The intermediate definition of $\bar{H} =  \sum_{t=1,i=1}^{T-1,n} \bar{h}^{0}_{ti}$ leads to} 
    \Delta_Q &= (T n \epsilon J_s) \eta_1 + (\bar{H} J_s) \eta_0.
\end{align}
\end{subequations}
\begin{align}
        \intertext{The spectral radius of $\Delta_Q$ is}
    \rho(\Delta_Q) &= (T n \epsilon s) \eta_1 + (\bar{H} s) \eta_0.
\end{align}
By Lemma \ref{lem:interval}, a sufficient condition for the interval family in \eqref{eq:interval_Q} to be \ac{PD} is if
\begin{align}
     &(\lambda^* - (T n \epsilon) \eta_2 )I - [(T n \epsilon s) \eta_1 + (\bar{H} s) \eta_0]I \in \psd_{++}^s,
     \intertext{which may be equivalently expressed as}
     & \lambda^* > (T n \epsilon) \eta_2 + (T n \epsilon s) \eta_1 + (\bar{H} s) \eta_0 > 0.
     \label{eq:lambda_eta_joint}
\end{align}

The combined conditions for admissible $(\eta_0, \eta_1, \eta_2)$ are
\begin{subequations}
\label{eq:eta_all}
\begin{align}
    & \eta_0, \eta_1, \eta_2 > 0, \qquad    \eta_2 > s \eta_1 \\
    & \lambda^* > (T n \epsilon) \eta_2+ (T n \epsilon s) \eta_1 + (\bar{H} s) \eta_0.
\end{align}
\end{subequations}

One possible choice of $(\eta_0, \eta_1, \eta_2)>0$ satisfying \eqref{eq:eta_all} is
\begin{align}
\label{eq:eta_choice}
    \eta_0 &= \frac{\lambda^*}{4 \bar{H} s} & \eta_1 &= \frac{\lambda^*}{ 2 T n \epsilon (2s + 1)} & \eta_2 &= \frac{\lambda^* (s+1)}{2 T n \epsilon (2s + 1)}. 
\end{align}

A polynomial multiplier certificate $(\zeta^\pm, \mu)$ will therefore always exist whenever \eqref{eq:altern_feas} is satisfied.

%% file: noise_arxiv.bbl
\begin{thebibliography}{10}
\providecommand{\url}[1]{#1}
\csname url@samestyle\endcsname
\providecommand{\newblock}{\relax}
\providecommand{\bibinfo}[2]{#2}
\providecommand{\BIBentrySTDinterwordspacing}{\spaceskip=0pt\relax}
\providecommand{\BIBentryALTinterwordstretchfactor}{4}
\providecommand{\BIBentryALTinterwordspacing}{\spaceskip=\fontdimen2\font plus
\BIBentryALTinterwordstretchfactor\fontdimen3\font minus
  \fontdimen4\font\relax}
\providecommand{\BIBforeignlanguage}[2]{{%
\expandafter\ifx\csname l@#1\endcsname\relax
\typeout{** WARNING: IEEEtran.bst: No hyphenation pattern has been}%
\typeout{** loaded for the language `#1'. Using the pattern for}%
\typeout{** the default language instead.}%
\else
\language=\csname l@#1\endcsname
\fi
#2}}
\providecommand{\BIBdecl}{\relax}
\BIBdecl

\bibitem{polyak2001optimal}
B.~Polyak and M.~Halpern, ``Optimal design for discrete-time linear systems via
  new performance index,'' \emph{International Journal of Adaptive Control and
  Signal Processing}, vol.~15, no.~2, pp. 129--152, 2001.

\bibitem{polyak2002superstable}
B.~T. Polyak and P.~S. Shcherbakov, ``Superstable linear control systems. i.
  analysis,'' \emph{Automation and Remote Control}, vol.~63, no.~8, pp.
  1239--1254, 2002.

\bibitem{barmish1985necessary}
B.~R. Barmish, ``Necessary and sufficient conditions for quadratic
  stabilizability of an uncertain system,'' \emph{Journal of Optimization
  theory and applications}, vol.~46, no.~4, pp. 399--408, 1985.

\bibitem{farina2000positive}
L.~Farina and S.~Rinaldi, \emph{{Positive Linear Systems: Theory and
  Applications}}.\hskip 1em plus 0.5em minus 0.4em\relax John Wiley \& Sons,
  2000, vol.~50.

\bibitem{lasserre2009moments}
J.~B. Lasserre, \emph{{Moments, Positive Polynomials And Their Applications}},
  ser. Imperial College Press Optimization Series.\hskip 1em plus 0.5em minus
  0.4em\relax World Scientific Publishing Company, 2009.

\bibitem{de2019formulas}
C.~De~Persis and P.~Tesi, ``{Formulas for Data-Driven Control: Stabilization,
  Optimality, and Robustness},'' \emph{IEEE Transactions on Automatic Control},
  vol.~65, no.~3, pp. 909--924, 2019.

\bibitem{willems2005note}
J.~C. Willems, P.~Rapisarda, I.~Markovsky, and B.~L. De~Moor, ``A note on
  persistency of excitation,'' \emph{Systems \& Control Letters}, vol.~54,
  no.~4, pp. 325--329, 2005.

\bibitem{coulson2019data}
J.~Coulson, J.~Lygeros, and F.~D{\"o}rfler, ``{Data-enabled predictive control:
  In the shallows of the DeePC},'' in \emph{2019 18th European Control
  Conference (ECC)}.\hskip 1em plus 0.5em minus 0.4em\relax IEEE, 2019, pp.
  307--312.

\bibitem{coulson2022robust}
J.~Coulson, H.~van Waarde, and F.~D{\"o}rfler, ``{Robust Fundamental Lemma for
  Data-driven Control},'' \emph{arXiv preprint arXiv:2205.06636}, 2022.

\bibitem{shafai2022data}
B.~Shafai, A.~Moradmand, and M.~Siami, ``{Data-Driven Positive Stabilization of
  Linear Systems},'' in \emph{2022 8th International Conference on Control,
  Decision and Information Technologies (CoDIT)}, vol.~1.\hskip 1em plus 0.5em
  minus 0.4em\relax IEEE, 2022, pp. 1031--1036.

\bibitem{van2020noisy}
H.~J. van Waarde, M.~K. Camlibel, and M.~Mesbahi, ``From noisy data to feedback
  controllers: non-conservative design via a matrix {S}-lemma,'' \emph{IEEE
  Transactions on Automatic Control}, 2020.

\bibitem{bisoffi2021data}
A.~Bisoffi, C.~De~Persis, and P.~Tesi, ``Data-driven control via {P}etersen's
  lemma,'' \emph{arXiv preprint arXiv:2109.12175}, 2021.

\bibitem{berberich2020combining}
J.~Berberich, C.~W. Scherer, and F.~Allg{\"o}wer, ``Combining prior knowledge
  and data for robust controller design,'' \emph{IEEE Transactions on Automatic
  Control}, 2022.

\bibitem{bianchi2022data}
M.~Bianchi, S.~Grammatico, and J.~Cort{\'e}s, ``Data-driven stabilization of
  switched and constrained linear systems,'' \emph{arXiv preprint
  arXiv:2208.11392}, 2022.

\bibitem{van2020informativity}
H.~J. van Waarde, J.~Eising, H.~L. Trentelman, and M.~K. Camlibel, ``Data
  informativity: A new perspective on data-driven analysis and control,''
  \emph{IEEE Transactions on Automatic Control}, vol.~65, no.~11, pp.
  4753--4768, 2020.

\bibitem{dai2020data}
T.~Dai, M.~Sznaier, and B.~R. Solvas, ``{Data-Driven Quadratic Stabilization of
  Continuous LTI Systems},'' \emph{IFAC-PapersOnLine}, vol.~53, no.~2, pp.
  3965--3970, 2020.

\bibitem{cheng2015}
Y.~Cheng, M.~Sznaier, and C.~Lagoa, ``{Robust Superstabilizing Controller
  Design from Open-Loop Experimental Input/Output Data},''
  \emph{IFAC-PapersOnLine}, vol.~48, no.~28, pp. 1337--1342, 2015, 17th IFAC
  Symposium on System Identification SYSID 2015.

\bibitem{dai2018data}
T.~Dai and M.~Sznaier, ``{Data Driven Robust Superstable Control of Switched
  Systems},'' \emph{IFAC-PapersOnLine}, vol.~51, no.~25, pp. 402--408, 2018.

\bibitem{dai2022convex}
------, ``A convex optimization approach to synthesizing state feedback
  data-driven controllers for switched linear systems,'' \emph{Automatica},
  vol. 139, p. 110190, 2022.

\bibitem{miller2023ddcpos}
J.~Miller, T.~Dai, M.~Sznaier, and B.~Shafai, ``{Data-Driven Control of
  Positive Linear Systems using Linear Programming},'' 2023,
  \href{https://arxiv.org/abs/2303.12242}{arxiv:2303.12242}.

\bibitem{norton1987identification}
J.~Norton, ``Identification of parameter bounds for {ARMAX} models from records
  with bounded noise,'' \emph{International Journal of control}, vol.~45,
  no.~2, pp. 375--390, 1987.

\bibitem{cerone1993feasible}
V.~Cerone, ``Feasible parameter set for linear models with bounded errors in
  all variables,'' \emph{Automatica}, vol.~29, no.~6, pp. 1551--1555, 1993.

\bibitem{cerone2011set}
V.~Cerone, D.~Piga, and D.~Regruto, ``{Set-Membership Error-in-Variables
  Identification Through Convex Relaxation Techniques},'' \emph{IEEE
  Transactions on Automatic Control}, vol.~57, no.~2, pp. 517--522, 2011.

\bibitem{soderstrom2018errors}
T.~S{\"o}derstr{\"o}m, \emph{{Errors-in-Variables Methods in System
  Identification}}.\hskip 1em plus 0.5em minus 0.4em\relax Springer, 2018.

\bibitem{cerone1993parameter}
V.~Cerone, ``Parameter bounds for {ARMAX} models from records with bounded
  errors in variables,'' \emph{International Journal of Control}, vol.~57,
  no.~1, pp. 225--235, 1993.

\bibitem{SZNAIER19963550}
M.~Sznaier, R.~Suárez, S.~Miani, and J.~Alvarez-Ramírez, ``Optimal $l_\infty$
  disturbance rejection and global stabilization of linear systems with
  saturating control,'' \emph{IFAC Proceedings Volumes}, vol.~29, no.~1, pp.
  3550--3555, 1996, 13th World Congress of IFAC, 1996, San Francisco USA, 30
  June - 5 July.

\bibitem{polyak2004extended}
B.~T. Polyak, ``Extended superstability in control theory,'' \emph{Automation
  and Remote Control}, vol.~65, no.~4, pp. 567--576, 2004.

\bibitem{rami2007controller}
M.~A. Rami and F.~Tadeo, ``{Controller Synthesis for Positive Linear Systems
  With Bounded Controls},'' \emph{IEEE Transactions on Circuits and Systems II:
  Express Briefs}, vol.~54, no.~2, pp. 151--155, 2007.

\bibitem{aubin2009set}
J.-P. Aubin and H.~Frankowska, \emph{{Set-Valued Analysis}}.\hskip 1em plus
  0.5em minus 0.4em\relax Springer Science \& Business Media, 2009.

\bibitem{mangasarian1987lipschitz}
O.~L. Mangasarian and T.-H. Shiau, ``{Lipschitz Continuity of Solutions of
  Linear Inequalities, Programs and Complementarity Problems},'' \emph{SIAM
  Journal on Control and Optimization}, vol.~25, no.~3, pp. 583--595, 1987.

\bibitem{stone1948generalized}
M.~H. Stone, ``The generalized weierstrass approximation theorem,''
  \emph{Mathematics Magazine}, vol.~21, no.~5, pp. 237--254, 1948.

\bibitem{scherer2006matrix}
C.~W. Scherer and C.~W. Hol, ``{Matrix Sum-of-Squares Relaxations for Robust
  Semi-Definite Programs},'' \emph{Mathematical programming}, vol. 107, no.~1,
  pp. 189--211, 2006.

\bibitem{ben2015deriving}
A.~Ben-Tal, D.~Den~Hertog, and J.-P. Vial, ``Deriving robust counterparts of
  nonlinear uncertain inequalities,'' \emph{Mathematical programming}, vol.
  149, no. 1-2, pp. 265--299, 2015.

\bibitem{boyd2004convex}
S.~Boyd, S.~P. Boyd, and L.~Vandenberghe, \emph{{Convex Optimization}}.\hskip
  1em plus 0.5em minus 0.4em\relax Cambridge university press, 2004.

\bibitem{ben2002interval}
A.~Ben-Tal and A.~Nemirovski, ``On tractable approximations of uncertain linear
  matrix inequalities affected by interval uncertainty,'' \emph{SIAM Journal on
  Optimization}, vol.~12, no.~3, pp. 811--833, 2002.

\bibitem{zhen2022robust}
J.~Zhen, F.~J. de~Ruiter, E.~Roos, and D.~den Hertog, ``Robust optimization for
  models with uncertain second-order cone and semidefinite programming
  constraints,'' \emph{INFORMS Journal on Computing}, vol.~34, no.~1, pp.
  196--210, 2022.

\bibitem{caverly2019lmi}
R.~J. Caverly and J.~R. Forbes, ``{LMI Properties and Applications in Systems,
  Stability, and Control Theory},'' \emph{arXiv preprint arXiv:1903.08599},
  2019.

\bibitem{feron1992numerical}
E.~Feron, V.~Balakrishnan, S.~Boyd, and L.~El~Ghaoui, ``Numerical methods for h
  2 related problems,'' in \emph{1992 American control conference}.\hskip 1em
  plus 0.5em minus 0.4em\relax IEEE, 1992, pp. 2921--2922.

\bibitem{mosek92}
\BIBentryALTinterwordspacing
M.~ApS, \emph{The MOSEK optimization toolbox for MATLAB manual. Version 9.2.},
  2020. [Online]. Available:
  \url{https://docs.mosek.com/9.2/toolbox/index.html}
\BIBentrySTDinterwordspacing

\bibitem{Lofberg2004}
J.~L{\"{o}}fberg, ``{YALMIP : A Toolbox for Modeling and Optimization in
  MATLAB},'' in \emph{In Proceedings of the CACSD Conference}, Taipei, Taiwan,
  2004.

\bibitem{waki2006sums}
H.~Waki, S.~Kim, M.~Kojima, and M.~Muramatsu, ``{Sums of Squares and
  Semidefinite Programming Relaxations for Polynomial Optimization Problems
  with Structured Sparsity},'' \emph{SIOPT}, vol.~17, no.~1, pp. 218--242,
  2006.

\bibitem{parrilo2000structured}
P.~A. Parrilo, \emph{Structured semidefinite programs and semialgebraic
  geometry methods in robustness and optimization}.\hskip 1em plus 0.5em minus
  0.4em\relax California Institute of Technology, 2000.

\bibitem{hilbert1888darstellung}
D.~Hilbert, ``{\"U}ber die darstellung definiter formen als summe von
  formenquadraten,'' \emph{Mathematische Annalen}, vol.~32, no.~3, pp.
  342--350, 1888.

\bibitem{blekherman2006there}
G.~Blekherman, ``{There are Significantly More Nonnegative Polynomials than
  Sums of Squares},'' \emph{Israel Journal of Mathematics}, vol. 153, no.~1,
  pp. 355--380, 2006.

\bibitem{putinar1993compact}
M.~Putinar, ``{Positive Polynomials on Compact Semi-algebraic Sets},''
  \emph{Indiana University Mathematics Journal}, vol.~42, no.~3, pp. 969--984,
  1993.

\bibitem{cimpric2011closures}
J.~Cimpri{\v{c}}, M.~Marshall, and T.~Netzer, ``Closures of quadratic
  modules,'' \emph{Israel Journal of Mathematics}, vol. 183, no.~1, pp.
  445--474, 2011.

\bibitem{nie2007complexity}
J.~Nie and M.~Schweighofer, ``On the complexity of {Putinar's
  Positivstellensatz},'' \emph{Journal of Complexity}, vol.~23, no.~1, pp.
  135--150, 2007.

\bibitem{alizadeh1995interior}
F.~Alizadeh, ``{Interior Point Methods in Semidefinite Programming with
  Applications to Combinatorial Optimization},'' \emph{SIAM J OPTIMIZ}, vol.~5,
  no.~1, pp. 13--51, 1995.

\bibitem{denel1979extensions}
J.~Denel, ``Extensions of the continuity of point-to-set maps: {A}pplications
  to fixed point algorithms,'' \emph{Point-to-Set Maps and Mathematical
  Programming}, pp. 48--68, 1979.

\bibitem{daniilidis2013lower}
A.~Daniilidis, M.~A. Goberna, M.~A. L{\'o}pez, and R.~Lucchetti, ``Lower
  semicontinuity of the feasible set mapping of linear systems relative to
  their domains,'' \emph{Set-Valued and Variational Analysis}, vol.~21, no.~1,
  pp. 67--92, 2013.

\bibitem{rohn1994positive}
J.~Rohn, ``Positive definiteness and stability of interval matrices,''
  \emph{SIAM Journal on Matrix Analysis and Applications}, vol.~15, no.~1, pp.
  175--184, 1994.

\end{thebibliography}
